\documentclass[a4paper,12pt]{article}

\usepackage{hyperref}

\usepackage{fullpage}
\usepackage{graphicx}
\usepackage{amsmath}
\usepackage{amssymb}
\usepackage{subfigure}
\usepackage{epsfig}

\usepackage{caption,color}
\usepackage{amsthm}

\usepackage{bm}

\usepackage{url}

\usepackage{accents}

\theoremstyle{plain}
\newtheorem{thm}{Theorem}[section]
\newtheorem{cor}[thm]{Corollary}
\newtheorem{lem}[thm]{Lemma}
\newtheorem{prop}[thm]{Proposition}
\newtheorem*{propA}{Proposition A}
\theoremstyle{remark}

\author{
Chiun-Chuan Chen$^{1,2}$ and Li-Chang Hung$^{1}$\footnote{Corresponding author's email address: \texttt{lichang.hung@gmail.com}
         }
\vspace{10mm}
\\
 \small
 \textit{$^{1}$ Department of Mathematics, National Taiwan University, Taiwan}
\\
\small
\textit{$^{2}$ National Center for Theoretical Sciences, Taiwan}
 \small
}

\title{A maximum principle for diffusive Lotka-Volterra systems of two competing species
}

\date{
\small
Li-Chang Hung dedicates this work to Mach Nguyet Minh
}

\begin{document}
\maketitle

\begin{abstract}

Using an elementary approach, we establish a new maximum principle for the diffusive Lotka-Volterra system of two competing species, which involves
pointwise estimate of an elliptic equation consisting of the second derivative of one function, the first derivative of another function,
and a quadratic nonlinearity. This maximum principle gives a priori estimates for the total mass of the two species.
Moreover, applying it to the system of three competing species leads to a nonexistence theorem of traveling wave solutions.










\end{abstract}



\section{Introduction}
\vspace{2mm}

In this paper we study the following diffusive Lotka-Volterra system of two competing species:
\begin{equation}\label{eqn: compet L-V sys of 2 species with diffu}
\begin{cases}
u_t=d_1\,u_{yy}+u\,(\sigma_1-c_{11}\,u-c_{12}\,v), \quad y\in\mathbb{R},\quad t>0,\\ \\
\hspace{0.7mm}v_t=d_2\,v_{yy}+v\,(\sigma_2-c_{21}\,u-c_{22}\,v), \quad y\in\mathbb{R},\quad t>0,\\
\end{cases}
\end{equation}
which is a system frequently used to model competitive behaviour between two distinct species. Here $u(y,t)$ and $v(y,t)$ stand for the density of the two species $u$ and $v$, respectively; $d_i$, $\sigma_i$, $c_{ii}$ $(i=1,2)$, and $c_{ij}$ $(i,j=1,2  \ \text{with} \;i\neq j)$ are the respective diffusion rates, intrinsic growth rates, intra-specific competition rates, and inter-specific competition rates, all of which are assumed to be positive.
The problem as to which species will survive in a competitive system is of importance in ecology. In order to tackle this problem,
we consider traveling wave solutions, which are solutions of the form
\begin{equation}\label{eqn: traveling wave (u,v)=(U,V)}
(u(y,t),v(y,t))=(u(x),v(x)), \quad x=y-\theta \,t,
\end{equation}
where $\theta$ is the propagation speed of the traveling wave. In general, the sign of $\theta$ indicates which species is stronger and can survive.




We note that by using a suitable scaling, the two-species system \eqref{eqn: compet L-V sys of 2 species with diffu} can be rewritten as
\begin{equation}\label{eqn: compet L-V sys of 2 species with diffu scaled}
\begin{cases}
u_t=\;\;u_{yy}+\;\;u\,(1-u-a_1\,v), \quad y\in\mathbb{R},\quad t>0,\\ \\
v_t=d\,v_{yy}+k\,v\,(1-a_2\,u-v),  \quad y\in\mathbb{R},\quad t>0,\\
\end{cases}
\end{equation}
where $d$, $k$, $a_1$ and $a_2$ are positive parameters. It is readily seen that in general, \eqref{eqn: compet L-V sys of 2 species with diffu scaled} has four equilibria: $\textbf{e}_1=(0,0)$, $\textbf{e}_2=(1,0)$, $\textbf{e}_3=(0,1)$ and $\textbf{e}_4=(u^{\ast},v^{\ast})$, where $(u^{\ast},v^{\ast})=(\frac{1-a_1}{1-a_1\,a_2},\frac{1-a_2}{1-a_1\,a_2})$ is the intersection of the two straight lines $1-u-a_{1}\,v=0$ and $1-a_{2}\,u-v=0$, whenever it exists. We note that $u^{\ast},v^{\ast}>0$ if and only if $a_1,a_2<1$ or $a_1,a_2>1$. When the domain is bounded, the asymptotic behavior of solutions $(u(y,t),v(y,t))$ for \eqref{eqn: compet L-V sys of 2 species with diffu scaled} with initial conditions $u(y,0),v(y,0)>0$ can be classified into four cases, as described in:

\vspace{2mm}

\begin{propA}[\cite{demottoni79}]\label{prop: ODE stability of two-species systems}
Let $(u(y,t),v(y,t))$ be the solution of \eqref{eqn: compet L-V sys of 2 species with diffu scaled} with the entire space $\mathbb{R}$ replaced by a bounded domain in $\mathbb{R}$ under the zero Neumann boundary conditions. Then for initial conditions $u(x,0)$,$v(x,0)>0$, we have
\begin{itemize}
\item [(i)]   $a_1<1<a_2$ \hspace{1.0mm} \ \ \ \ \ $\Rightarrow$
                       $\lim\limits^{}_{t\rightarrow\infty}(u(y,t),v(y,t))=(1,0)$;
\item [(ii)]  $a_2<1<a_1$ \hspace{1.0mm} \ \ \ \ \ $\Rightarrow$
                       $\lim\limits^{}_{t\rightarrow\infty}(u(y,t),v(y,t))=(0,1)$;
\item [(iii)] $a_1>1,a_2>1$\hspace{5.5mm} $\Rightarrow$
                       $(1,0)$ and $(0,1)$ are locally stable equilibria;
\item [(iv)]  $a_1<1,a_2<1$\hspace{5.5mm} $\Rightarrow$
                       $\lim\limits^{}_{t\rightarrow\infty}(u(y,t),v(y,t))=(u^{\ast},v^{\ast})$.
                         \end{itemize}
\end{propA}

\vspace{2mm}

In this paper, we consider the following traveling wave problems, which are obtained by substituting \eqref{eqn: traveling wave (u,v)=(U,V)}
into \eqref{eqn: compet L-V sys of 2 species with diffu scaled} and into \eqref{eqn: compet L-V sys of 2 species with diffu} respectively,
\begin{equation}\label{eqn: L-V BVP scaled}
\begin{cases}
\vspace{3mm}
\hspace{2.0mm} u_{xx}+\theta\,u_{x}+\;\;u\,(1-u-a_1\,v)=0, \quad x\in\mathbb{R}, \\
\vspace{3mm}
d\,v_{xx}+\theta\,v_{x}+k\,v\,(1-a_2\,u-v)=0, \quad x\in\mathbb{R},\\
(u,v)(-\infty)=\text{\bf e}_2,\quad (u,v)(+\infty)= \text{\bf e}_3,
\end{cases}
\end{equation}
and
\begin{equation}\label{eqn: L-V BVP  before scaled}
\begin{cases}
\vspace{3mm}
d_1\,u_{xx}+\theta\,u_{x}+u\,(\sigma_1-c_{11}\,u-c_{12}\,v)=0, \quad x\in\mathbb{R}, \\
\vspace{3mm}
d_2\,v_{xx}\hspace{0.8mm}+\theta\,v_{x}+v\,(\sigma_2-c_{21}\,u-c_{22}\,v)=0,\quad x\in\mathbb{R}, \\
(u,v)(-\infty)=\big(\frac{\displaystyle\sigma_1}{\displaystyle c_{11}},0\big),\quad
(u,v)(+\infty)=\big(0,\frac{\displaystyle\sigma_2}{\displaystyle c_{22}}\big).
\end{cases}
\end{equation}
We call a solution $(u(x),v(x))$ of \eqref{eqn: L-V BVP scaled} an $(\textbf{e}_2,\textbf{e}_3)$-wave. The typical situation
\begin{equation}
\lim_{x\to-\infty}(u,v)(x)=(1,0), \; \lim_{x\to\infty}(u,v)(x)=(0,1),
\end{equation}
i.e. $u$ is dominant on the left region and $v$ is dominant on the right region of $\mathbb R$, motivates us to study the
$(\textbf{e}_2,\textbf{e}_3)$-wave. In this situation, $u$ will occupy the whole domain  eventually if $\theta>0$
while  $v$ will occupy the whole domain  eventually if $\theta<0$. From the viewpoint of ecology, we can conclude
that the sign of $\theta$ determines which species is stronger, i.e. $u$ is stronger if $\theta>0$ and $v$
is stronger if $\theta<0$.

Much attention has been paid to the $(\textbf{e}_2,\textbf{e}_3)$-wave. For cases $(\textit{i})$ or $(\textit{iii})$
in Proposition A,
Kan-on (\cite{Kan-on95},\cite{Kan-on97Fisher-Monostable}),
Fei and Carr (\cite{Fei&Carr03}), Leung, Hou and Li (\cite{Leung08}), and Leung and Feng (\cite{Leung&Feng11}) established
the existence of $(\textbf{e}_2,\textbf{e}_3)$-waves employing different approaches. Under certain assumptions on the parameters,
Mimura and Rodrigo (\cite{Rodrigo&Mimura00,Rodrigo&Mimura01}) constructed exact $(\textbf{e}_2,\textbf{e}_3)$-waves by applying
a judicious ans\"{a}tz for solutions. By applying the hyperbolic tangent method, Hung(\cite{hung2012JJIAM}) found exact
$(\textbf{e}_2,\textbf{e}_3)$-waves under certain assumptions on the parameters.  All the exact $(\textbf{e}_2,\textbf{e}_3)$-waves
proposed by Mimura and Rodrigo (\cite{Rodrigo&Mimura00,Rodrigo&Mimura01}) and Hung(\cite{hung2012JJIAM}) are represented in terms of
polynomials in hyperbolic tangent functions. Throughout this paper, we restrict our attention to the \textit{bistable case},
i.e. case $(\textit{iii})$ $a_1>1,a_2>1$ in Proposition A.

To understand the ecological capacity of the inhabitant of the two competing species, the investigation of the total mass or the total density
of the two species $u$ and $v$ is essential since the inhabitant is resource-limited. This gives rises to the problem as to the estimate
of $u+v$ in \eqref{eqn: L-V BVP  before scaled}. In \cite{Chen&Hung15Nonexistence}, upper and lower bounds of $u+v$ are given
when the two diffusion rates $d_1$ and $d_2$ are equal. However, the approach employed in \cite{Chen&Hung15Nonexistence} to obtain upper or lower bounds
for $u+v$ cannot be applied to the case where the diffusion rates $d_1$ and $d_2$ are not equal.

The above discussion raises the following questions:
\vspace{2mm}

\textbf{Q1}: \textit{In \eqref{eqn: L-V BVP  before scaled}, when $d_1\neq d_2$, can upper and lower bounds of $u+v$ be obtained?}

\vspace{2mm}

As for the answer to \textbf{Q1}, it seems as far as we know, not available in the literature. To give an affirmative answer to this question, we develop a new but elementary approach. In fact, employing this approach leads to an affirmative answer to the following more general question:

\vspace{2mm}

\textbf{Q2}: \textit{In \eqref{eqn: L-V BVP  before scaled}, when $d_1\neq d_2$, can upper and lower bounds of $\tau_1\,u+\tau_2\,v$, where $\tau_1,\tau_2>0$ are arbitrary constants, be given?}

\vspace{2mm}

Since the physical units of $u$ and $v$ may not be identical, it makes sense to consider $\tau_1\,u+\tau_2\,v$ for the total mass in general. Although we can estimate $\tau_1\,u+\tau_2\,v$ via
\begin{equation}
 \min[\tau_1,\tau_2]\,(u+v)\leq \tau_1\,u+\tau_2\,v \leq \max[\tau_1,\tau_2]\,(u+v)
\end{equation}
once $u+v$ is measured, more information will be wasted in this manner of estimation as the difference of $\tau_1$ and $\tau_2$ becomes larger. Consequently, an approach which can accommodate to various $\tau_1$ and $\tau_2$ will be of great interest.


By adding the two equations in \eqref{eqn: L-V BVP  before scaled}, we obtain an equation involving $p(x)=\alpha\,u+\beta\,v$ and $q(x)=d_1\,\alpha\,u+d_2\,\beta\,v$
\begin{align}\label{eq: q''+p'+f+g=0 before scale}
  0&=\alpha\,\big(d_1\,u_{xx}+\theta\,u_{x}+u\,(\sigma_1-c_{11}\,u-c_{12}\,v)\big)+
      \beta\,\big(d_2\,v_{xx}+\theta\,v_{x}+v\,(\sigma_2-c_{21}\,u-c_{22}\,v)\big)\notag\\[3mm]
  &=q''(x)+\theta\,p'(x)+
  \alpha\,u\,(\sigma_1-c_{11}\,u-c_{12}\,v)+
    \beta\,v\,(\sigma_2-c_{21}\,u-c_{22}\,v).\notag\\
\end{align}
The case where $d_1=d_2$ or $p(x)$ is a constant multiple of $q(x)$ has been considered in \cite{Chen&Hung15Nonexistence}.
Obviously, difficulties arise, and the approach used in \cite{Chen&Hung15Nonexistence} cannot be applied
when $d_1\neq d_2$ since $p(x)$ no longer can be written as a constant multiple of $q(x)$.
The approach proposed here can be employed to give estimates of $q(x)$ even distinct $p(x)$ and $q(x)$ (i.e. $d_1\neq d_2$) are involved
in the \textit{scalar} equation \eqref{eq: q''+p'+f+g=0 before scale}.

To simplify the problem, we consider \eqref{eqn: L-V BVP scaled} first and present the results for \eqref{eqn: L-V BVP  before scaled}
in Section~\ref{sec: appendix}. One of the main results in this paper is the following maximum principle for Lotka-Volterra systems
of two strongly competing species.



\vspace{2mm}






\begin{thm}[\textbf{Maximum Principle for $q(x)$}]\label{thm: lb<q(x)<ub}
Suppose that $a_1>1,a_2>1$ and $(u(x),v(x))$ is a nonnegative solution to \eqref{eqn: L-V BVP scaled}. Then
\begin{equation}
\min\Big[\frac{\alpha}{a_2\,d},\frac{\beta}{a_1}\Big]\min[1,d^2]\leq q(x)\leq \max\Big[\frac{\alpha}{d},\beta\Big]\max[1,d^2] \ \text{ for }  x\in\mathbb{R},
\end{equation}
where $q(x)=\alpha\,u(x)+d\,\beta\,v(x)$ and $\alpha$, $\beta$ are arbitrary positive constants.
\end{thm}

\vspace{2mm}

\noindent In particular, we notice that the estimate of $q$ in Theorem~\ref{thm: lb<q(x)<ub} does not depend on the propagating speed $\theta$ and
the constant $k$.

\vspace{2mm}

The maximum principle in Theorem~\ref{thm: lb<q(x)<ub} can be generalized to hold true for a wider class of autonomous elliptic systems:
\begin{equation}\label{eqn: L-V BVP  before scaled-G}
\begin{cases}
\vspace{3mm}
d_1\,u_{xx}+\theta\,u_{x}+u^m\,f(u,v)=0, \quad x\in\mathbb{R}, \\
\vspace{3mm}
d_2\,v_{xx}\hspace{0.8mm}+\theta\,v_{x}+v^n\,g(u,v)=0,\quad x\in\mathbb{R}, \\
(u,v)(-\infty)=\text{\bf e}_{-},\quad (u,v)(+\infty)= \text{\bf e}_{+},
\end{cases}
\end{equation}
where $m\ge 0$, $n\ge 0$ and
$$
\text{\bf e}_{-}, \text{\bf e}_{+}\in \mathcal{C}_{f,g}=\{ (u,v) \;|\; u^mf(u,v)=0, v^ng(u,v)=0, u,v\ge 0\}.
$$
We assume that $f(u,v)\in C^{0,\tau}(\mathbb{R^{+}}\times\mathbb{R^{+}})$ and $g(u,v)\in C^{0,\tau}(\mathbb{R^{+}}\times\mathbb{R^{+}})$ for some $\tau>0$,
and the following property holds:

\begin{itemize}
\item [$\mathbf{[A]}$] There exist $\bar{u}>\underaccent\bar{u}>0$ and $\bar{v}>\underaccent\bar{v}>0$ such that
\begin{eqnarray*}
f(u,v)\le 0 \text{ and } g(u,v)\le 0 &\text{ if } (u,v)\in \bar{\mathcal{R}}=\big\{ (u,v)\;\big|\; \frac{u}{\bar{u}}+\frac{v}{\bar{v}}\ge 1,\; u,v\geq0 \big\};\\
f(u,v)\ge 0 \text{ and } g(u,v)\ge 0 &\text{ if } (u,v)\in\underaccent\bar{\mathcal{R}}=\big\{ (u,v)\;\big|\; \frac{u}{\underaccent\bar{u}}+\frac{v}{\underaccent\bar{v}}\le1,\;u,v\geq0\big\}.
\end{eqnarray*}
\end{itemize}

\noindent We have the following theorem.
\vspace{3mm}

\begin{thm}[\textbf{Generalized Maximum Principle}]\label{thm: N-Barrier Maximum Principle for 2 Species}
Assume that $\mathbf{[A]}$ holds. If $a>0$, $b>0$, and $(u(x),v(x))$ is a nonnegative solution to \eqref{eqn: L-V BVP  before scaled-G}
with $\text{\bf e}_{-}\neq (0,0)$ and $\text{\bf e}_{+}\neq (0,0)$, then
\begin{equation}\label{eqn: upper and lower bounds of p}
\min\big(a\,\underaccent\bar{u},b\,\underaccent\bar{v}\big)\,\frac{\displaystyle\min(d_1,d_2)}{\displaystyle\max(d_1,d_2)}
\leq a\, u(x)+b\, v(x) \leq
\max\big(a\,\bar{u},b\,\bar{v}\big)\,\frac{\max(d_1,d_2)}{\min(d_1,d_2)}
\end{equation}
for $x\in\mathbb R$.

\end{thm}

\vspace{2mm}

Using the properties of the nonlinear terms of \eqref{eqn: L-V BVP scaled} more delicately, one can obtain better but complicated estimates
for $u+v$. In the following, we just state an improved result for $d=k=1$ since the form of the lower bound obtained is simple in this case.
More general results are described in Section \ref{sec: tangent lines}.

\begin{thm}\label{thm: sharper lower bound for u+v}
Suppose $d=k=1$, $a_1>1,a_2>1$, and $(u(x),v(x))$ is a nonnegative solution to \eqref{eqn: L-V BVP scaled}. Then for $x\in\mathbb{R}$
\begin{equation}\label{eqn: delicate}
\frac{4}{a_1+a_2+2}\leq u(x)+v(x)\leq 1.
\end{equation}
\end{thm}

\vspace{2mm}

It is easy to see that the lower bound for $u+v$ obtained by Theorem~\ref{thm: lb<q(x)<ub} is
$\min [1/a_1, 1/a_2]$, which is smaller than or equal to $\frac 4{a_1+a_2+2}$ and is less sharp when $a_1,a_2>1$. Note that the lower bound
in (\ref{eqn: delicate}) approaches $1$ as $(a_1,a_2)$ approaches $(1,1)$.

As an application of Theorem~\ref{thm: lb<q(x)<ub}, we establish nonexistence of traveling waves solutions for the Lotka-Volterra system of three competing species, i.e.nonexistence of traveling solutions of
\begin{equation}\label{eqn: L-V systems of three species (TWS)}
\begin{cases}
\vspace{3mm}
d_1\,u_{xx}\hspace{-0.5mm}+\theta \,u_x\hspace{-0.5mm}+u\,(\,\sigma_1-c_{11}\,u-c_{12}\,v-c_{13}\,w\,)=0, \quad x\in\mathbb{R}, \\
\vspace{3mm}
d_2\,v_{xx}+\theta \,v_x+v\,(\,\sigma_2-c_{21}\,u-c_{22}\,v-c_{23}\,w\,)=0, \quad x\in\mathbb{R},\\
d_3\,w_{xx}\hspace{-1mm}+\theta \,w_x\hspace{-1.5mm}+\hspace{-0.5mm}w\,(\,\sigma_3-c_{31}\,u-c_{32}\,v-c_{33}\,w\,)=0, \quad x\in\mathbb{R},\\
\end{cases}
\end{equation}
where $u(x,t)$, $v(x,t)$ and $w(x,t)$ represent the density of the three species $u$, $v$ and $w$ respectively; $d_i$, $\sigma_i$, $c_{ii}$ $(i=1,2,3)$, and $c_{ij}$ $(i,j=1,2,3, i\neq j)$ are the diffusion rates, the intrinsic growth rates, the intra-specific competition rates, and the inter-specific competition rates, respectively. These constants are all assumed to be positive.

For \eqref{eqn: L-V systems of three species (TWS)}, existence of solutions with profiles of one-hump waves supplemented with the boundary conditions
\begin{equation}\label{eqn: BC CHMU}
(u,v,w)(-\infty)=\Big(\frac{\sigma_1}{c_{11}},0,0\Big), \quad (u,v,w)(\infty)=\Big(0,\frac{\sigma_2}{c_{22}},0\Big)
\end{equation}
is investigated in \cite{CHMU}. Here a one-hump wave is a traveling wave which consists of a forward front $v$, a backward front $u$, and a pulse $w$ in the middle. By finding exact solutions and using the numerical tracking method AUTO, the existence of one-hump waves for \eqref{eqn: L-V systems of three species (TWS)},\eqref{eqn: BC CHMU} is established under certain assumptions on the parameters (\cite{CHMU}).

On the other hand, nonexistence of solutions for \eqref{eqn: L-V systems of three species (TWS)} and \eqref{eqn: BC CHMU} is studied in \cite{Chen&Hung15Nonexistence} when the diffusion rates $d_1$, $d_2$, and $d_3$ are assumed to be identical. In \cite{Chen&Hung15Nonexistence}, a subtle structure of the competing system, which heavily relies on equal diffusivity, is employed. With the aid of Theorem~\ref{thm: lb<q(x)<ub} (or the extended version Theorem~\ref{thm: lb<q(x)<ub, before scaling, appendix} in Section~\ref{sec: appendix}), we give a much more general nonexistence of solutions for \eqref{eqn: L-V systems of three species (TWS)} and \eqref{eqn: BC CHMU} when the diffusion rates of the species are no longer the same.

\vspace{2mm}

\begin{thm}[\textbf{Nonexistence of 3-species wave}]\label{thm: Nonexistence 3 species}
Let $\phi_1=\sigma_1\,c_{33}-\sigma_3\,c_{13}$ and $\phi_2=\sigma_2\,c_{33}-\sigma_3\,c_{23}$. Assume that the following hypotheses hold:
\begin{itemize}
\item [$\mathbf{[H1]}$] $\phi_1,\phi_2>0$; 
\item [$\mathbf{[H2]}$] $c_{21}\,\phi_1>c_{11}\,\phi_2,c_{12}\,\phi_2>c_{22}\,\phi_1$;
\item [$\mathbf{[H3]}$] 
$\min\bigg[\frac{\displaystyle c_{31}\,\phi_2}{\displaystyle c_{21}\,d_2},\frac{\displaystyle c_{32}\,\phi_1}{\displaystyle c_{12}\,d_1}\bigg]\,\min \big[d_1^2,d_2^2\big]\geq\sigma_3\,c_{33}$.
\end{itemize}
Then \eqref{eqn: L-V systems of three species (TWS)} and \eqref{eqn: BC CHMU} has no positive solution
$(u(x),v(x),w(x))$.
\end{thm}

\vspace{2mm}




\textbf{Biological interpretation}: \textrm{Due to $\mathbf{[H2]}$, $u$ and $v$ are strongly competing in \eqref{eqn: nonexistence diff ineq <0} (see Section~\ref{sec: nonexistence}). However, we can find parameters such that $\mathbf{[BiS]}$ (see the Appendix in Section~\ref{sec: appendix}) which is slightly different from $\mathbf{[H2]}$ holds as well, i.e. $u$ and $v$ are also strongly competing in \eqref{eqn: L-V systems of three species (TWS)} as $w$ is absent. Moreover, it is easy to see that $\mathbf{[H3]}$ clearly holds if $\sigma_3$ is sufficiently small when other parameters are fixed. In conclusion, Theorem~\ref{thm: Nonexistence 3 species} asserts that, under certain conditions on the parameters, the three species $u$, $v$ and $w$ in the ecological system modeled by \eqref{eqn: L-V systems of three species (TWS)} and \eqref{eqn: BC CHMU} cannot coexist if the intrinsic growth rate $\sigma_3$ of $w$ is sufficiently small when strong competition between $u$ and $v$ occurs.
}

\vspace{5mm}






The remainder of this paper is organized as follows. Section~\ref{sec: main result} is devoted to  the proof of Theorem~\ref{thm: lb<q(x)<ub}.
Then we generalize Theorem~\ref{thm: lb<q(x)<ub} in Section~\ref{sec: General NBMP}. By using the tangent line to the quadratic curve
$\alpha\,u\,(1-u-a_1\,v)+\beta\,k\,v\,(1-a_2\,u-v)=0$, it is shown in Section~\ref{sec: tangent lines} that, under a certain condition
on the parameters, a stronger lower bound than the one given in Proposition~\ref{prop: lower bed} and Theorem~\ref{thm: lb<q(x)<ub} can be obtained.
Also, the proof of Theorem~\ref{thm: sharper lower bound for u+v} is presented in Section~\ref{sec: tangent lines}.
As an application of Theorem~\ref{thm: lb<q(x)<ub}, we establish Theorem~\ref{thm: Nonexistence 3 species} in Section~\ref{sec: nonexistence}.
Finally, we conclude the paper with corresponding results for \eqref{eqn: L-V BVP  before scaled} in the Appendix (Section~\ref{sec: appendix}).

\vspace{2mm}
\setcounter{equation}{0}
\setcounter{figure}{0}
\setcounter{subfigure}{0}
\section{Proof of Theorem~\ref{thm: lb<q(x)<ub}}\label{sec: main result}
\vspace{2mm}

In this section $p(x)=\alpha\,u(x)+\beta\,v(x)$ and $q(x)=\alpha\,u(x)+d\,\beta\,v(x)$, where $\alpha$ and $\beta$ are arbitrary positive constants. We begin with a useful lemma. 


\vspace{2mm}

\begin{lem}\label{lem: bistable then hyperbola}
Under the bistable condition $a_1>1$ and $a_2>1$, the quadratic curve $\alpha\,u\,(1-u-a_1\,v)+\beta\,k\,v\,(1-a_2\,u-v)=0$ is a hyperbola for
$\alpha>0$ and $\beta>0$.
\end{lem}
\begin{proof}
The discriminant of the quadratic curve $\alpha\,u\,(1-u-a_1\,v)+\beta\,k\,v\,(1-a_2\,u-v)=0$ is $(\alpha\,a_1+\beta\,k\,a_2)^2-4\,\alpha\,\beta\,k$. Since $a_1,a_2>1$, we have $(\alpha\,a_1+\beta\,k\,a_2)^2\geq 4\,\alpha\,\beta\,k\,a_1\,a_2> 4\,\alpha\,\beta\,k$. The positivity of the discriminant  gives the desired result.
\end{proof}

\vspace{2mm}

\noindent The lemma indicates that the quadratic curve
\begin{equation}\label{quadratic}
F(u,v):=\alpha\,u\,(1-u-a_1\,v)+\beta\,k\,v\,(1-a_2\,u-v)=0
\end{equation}
cannot either be an ellipse or a parabola under the bistable condition $a_1,a_2>1$. 

In Propositions~\ref{prop: lower bed} and \ref{prop: upper bed} below, we give a lower bound and an upper bound for $q(x)$, respectively. Combining the results in Propositions~\ref{prop: lower bed} and \ref{prop: upper bed}, we immediately obtain Theorem~\ref{thm: lb<q(x)<ub}.

\vspace{2mm}

\begin{prop} [\textbf{Lower bound for $q=q(x)$}]\label{prop: lower bed}
Let $a_1>1$ and $a_2>1$. Suppose that $(u(x),v(x))$ is $C^2$, nonnegative, and satisfies the following differential inequalities
and asymptotic behaviour:
\begin{equation}\label{eqn: L-V <0}
\begin{cases}
\vspace{3mm}
\hspace{2.0mm} u_{xx}+\theta\,u_{x}+\;\;u\,(1-u-a_1\,v)\leq0, \quad x\in\mathbb{R}, \\
\vspace{3mm}
d\,v_{xx}+\theta\,v_{x}+k\,v\,(1-a_2\,u-v)\leq0, \quad x\in\mathbb{R},\\
(u,v)(-\infty)=\text{\bf e}_2,\quad (u,v)(+\infty)= \text{\bf e}_3.
\end{cases}
\end{equation}
Then we have for $x\in\mathbb{R}$,
\begin{equation}\label{q lower bound}
q(x)\geq \min\bigg[\frac{\alpha}{a_2\,d},\frac{\beta}{a_1}\bigg]\,\min [1,d^2].
\end{equation}

\end{prop}

\begin{proof}
Let $\underaccent\bar{\mathcal{R}}=\{(u.v)\,|\,1-u-a_1\,v\ge 0, 1-a_2\,u-v\ge 0, u\ge 0, v\ge 0\}$. First we construct an appropriate \textit{N-barrier}
consisting of three lines $\alpha\,u+d\,\beta\,v=\lambda_2$, $\alpha\,u+\beta\,v=\eta$ and
$\alpha\,u+d\,\beta\,v=\lambda_1$, and chose $\lambda_1$, $\lambda_2$ and $\eta$ as large as possible such that
$Q_{\lambda_1}\subset P_{\eta}\subset Q_{\lambda_2}\subset \underaccent\bar{\mathcal{R}}$,
where $Q_{\lambda}=\{(u,v)\,|\,\alpha\,u+d\,\beta\,v\le\lambda, u\ge 0, v\ge 0\}$ and
$P_{\eta}=\{(u,v)\,|\,\alpha\,u+\beta\,v\le\eta, u\ge 0, v\ge 0\}$. Then we show that $\lambda_1$ can be taken
to equal the value on the right hand side of (\ref{q lower bound}) and $q(x)\ge\lambda_1$ can be verified via the structure of the N-barrier.

Now we illustrate how to construct \textit{the N-barrier} in detail. For the case of $d\geq1$ and $\beta\,a_2\,d\geq \alpha\,a_1$,
the N-barrier is constructed in the following three steps (see Figure~2.\ref{fig: d>1_beta_a_2d>alpha_a1}):
\begin{itemize}
  \item[(1)] \textit{The construction of the upper pink line}: we draw on the $uv$-plane the upper pink line $\alpha\,u+d\,\beta\,v=\lambda_2$
  which passes through $(\frac{1}{a_2},0)$. This gives $\lambda_2=\frac{\alpha}{a_2}$, and hence the upper pink line is represented by the equation
  $\alpha\,u+d\,\beta\,v=\frac{\alpha}{a_2}$. The $v$-coordinate of the $v$-intercept of
  $\alpha\,u+d\,\beta\,v=\frac{\alpha}{a_2}$ is $\frac{\alpha}{\beta\,a_2\,d}$, which is less than or equal to $\frac{1}{a_1}$ by the assumption
  $\beta\,a_2\,d\geq \alpha\,a_1$. This means that the $v$-coordinate of the $v$-intercept of $\alpha\,u+d\,\beta\,v=\frac{\alpha}{a_2}$ is below
  the $v$-coordinate of $v$-intercept of the $1-u-a_1\,v=0$.
  \item[(2)] \textit{The construction of the yellow line}: we let the the yellow line $\alpha\,u+\beta\,v=\eta$ start from
  $(0,\frac{\alpha}{\beta\,a_2\,d})$. This leads to $\eta=\frac{\alpha}{a_2\,d}$ and hence the yellow line is represented by the equation
  $\alpha\,u+\beta\,v=\frac{\alpha}{a_2\,d}$. The $u$-coordinate of the $u$-intercept of $\alpha\,u+\beta\,v=\frac{\alpha}{a_2\,d}$ is
  $\frac{1}{a_2\,d}$, which is less than or equal to $\frac{1}{a_2}$ by the assumption $d\geq1$. This means that the $u$-coordinate of the
  $u$-intercept of $\alpha\,u+\beta\,v=\frac{\alpha}{a_2\,d}$ is less than or equal to the $u$-coordinate of $u$-intercept of
  $\alpha\,u+d\,\beta\,v=\frac{\alpha}{a_2}$.
  \item[(3)]  \textit{The construction of the lower pink line}: we draw the lower pink line $\alpha\,u+d\,\beta\,v=\lambda_1$
  passing through $(\frac{1}{a_2\,d},0)$. This gives $\lambda_1=\frac{\alpha}{a_2\,d}$.
\end{itemize}
There are three other cases, each of which can be treated in a similar manner for the construction of the corresponding N-barrier (see Figures~2.\ref{fig: d>1_beta_a_2d<alpha_a1}, 2.\ref{fig: d<1_beta_a_2d>alpha_a1}, and 2.\ref{fig: d<1_beta_a_2d<alpha_a1}).
More precisely, we have the following four cases and for each case, we take different $\lambda_1$, $\lambda_2$ and $\eta$, and show that $q(x)$
has the lower bound $\lambda_1$ for $x\in\mathbb R$:
\begin{itemize}
     \item If $d\geq1$,
\begin{itemize}
\item [$(i)$] when $\beta\,a_2\,d\geq \alpha\,a_1$, we take $(\lambda_1,\lambda_2,\eta):=(\dfrac{\alpha}{a_2\,d},\dfrac{\alpha}{a_2},\dfrac{\alpha}{a_2\,d})$;
\item [$(ii)$] when $\beta\,a_2\,d < \alpha\,a_1$, we take $(\lambda_1,\lambda_2,\eta):=(\dfrac{\beta}{a_1},\dfrac{\beta\,d}{a_1},\dfrac{\beta}{a_1})$.
\end{itemize}
      \item If $d<1$,
\begin{itemize}
\item [$(iii)$] when $\beta\,a_2\,d\geq \alpha\,a_1$, $(\lambda_1,\lambda_2,\eta):=(\dfrac{\alpha\,d}{a_2},\dfrac{\alpha}{a_2},\dfrac{\alpha}{a_2})$;
\item [$(iv)$] when $\beta\,a_2\,d < \alpha\,a_1$, $(\lambda_1,\lambda_2,\eta):=(\dfrac{\beta\,d^2}{a_1},\dfrac{\beta\,d}{a_1},\dfrac{\beta\,d}{a_1})$.
\end{itemize}
We note that case $(i)$ corresponds to Figure~2.\ref{fig: d>1_beta_a_2d>alpha_a1},
in which the N-barrier has been constructed in the above three steps.
The other cases $(ii)$, $(iii)$, and $(iv)$ correspond to Figures~2.\ref{fig: d>1_beta_a_2d<alpha_a1}, 2.\ref{fig: d<1_beta_a_2d>alpha_a1},
and 2.\ref{fig: d<1_beta_a_2d<alpha_a1}, respectively.
\end{itemize}

We first observe that the property $q(x)\ge\lambda_1$ in the four cases can be reduced to the following two cases:
\begin{itemize}
  \item for $\beta\,a_2\,d\geq \alpha\,a_1$, $q(x)\geq \frac{\alpha}{a_2}\min [d,1/d]$ for all $x\in \mathbb{R}$;
  \item for $\beta\,a_2\,d<      \alpha\,a_1$, $q(x)\geq \frac{\beta}{a_1}\min [1,d^2]$ for all $x\in \mathbb{R}$.
\end{itemize}
Combining the two cases above leads to $q(x)\geq \min[\frac{\alpha}{a_2\,d},\frac{\beta}{a_1}]\,\min [1,d^2]$ for all $x\in \mathbb{R}$,
which is the desired result.

Now we show $q(x)\ge\lambda_1$ in $(i)\sim (iv)$. The two inequalities in \eqref{eqn: L-V <0} and (\ref{quadratic}) give
\begin{equation}\label{eqn: ODE for p and q<}
q''(x)+\theta\,p'(x)+F(u(x),v(x))\leq0.
\end{equation}
For $d>1$, we first prove $(i)$ by contradiction. Suppose that, contrary to our claim, there exists $z\in\mathbb{R}$ such that
$q(z)<\lambda_1$. Since $u,v\in C^2(\mathbb{R})$, by $(u,v)(-\infty)=(1,0)$ and $(u,v)(+\infty)=(0,1)$, we may assume $\min_{x\in\mathbb{R}} q(x)=q(z)$. 
We denote respectively by $z_2$ and $z_1$ the first points at which the solution $(u(x),v(x))$ intersects the line $\alpha\,u+d\,\beta\,v=\lambda_2$
in the $uv$-plane when $x$ moves from $z$ towards $\infty$ and $-\infty$ (as shown in Figure~2.\ref{fig: d>1_beta_a_2d>alpha_a1}). For the case where $\theta\leq0$, we integrate \eqref{eqn: ODE for p and q<} with respect to $x$ from $z_1$ to $z$ and obtain
\begin{equation}\label{eqn: integrating eqn}
q'(z)-q'(z_1)+\theta\,(p(z)-p(z_1))+\int_{z_1}^{z}F(u(x),v(x))\,dx\leq0.
\end{equation}
On the other hand we have:
\begin{itemize}
  \item since $\min_{x\in\mathbb{R}} q(x)=q(z)$, $q'(z)=\alpha\,u'(z)+d\,\beta\,v'(z)=0$;
  \item $q(z_1)=\lambda_2$ follows from the fact that $z_1$ is on the line $\alpha\,u+d\,\beta\,v=\lambda_2$. Since $z_1$ is the first point
  for $q(x)$ taking the value $\lambda_2$ when $x$ moves from $z$ to $-\infty$, we conclude that $q(z_1+\delta)\leq \lambda_2$ for $z-z_1>\delta>0$
  and $q'(z_1)\leq 0$;
  \item $p(z)<\eta$ since $z$ is below the line $\alpha\,u+\beta\,v=\eta$; $p(z_1)>\eta$ since $z$ is above the line $\alpha\,u+\beta\,v=\eta$;
  \item it is readily seen that the quadratic curve $F(u,v)=0$ passes through the points $(0,0)$, $(1,0)$, $(0,1)$, and $(u^{\ast},v^{\ast})$
  in the $uv$-plane. Let $A_+=\{(u,v)\,|\, F(u,v)\ge 0, u\ge 0,v\ge0\}$. By Lemma~\ref{lem: bistable then hyperbola} and the property that $F(u,v)<0$
  for large $u$ and $v$, it follows that $A_+$ is the region bounded by a hyperbola, $u$-axis and $v$-axis.
  Moreover, $\{(u(x),v(x))\,|\,z_1\le x \le z\}$ $\subset\underaccent\bar{\mathcal{R}}$ $\subset A_+$.
  Therefore we have $\int_{z_1}^{z}F(u(x),v(x))\,dx>0$.

\end{itemize}
Summarizing the above arguments, we obtain
\begin{equation}
q'(z)-q'(z_1)+\theta\,(p(z)-p(z_1))+\int_{z_1}^{z}F(u(x),v(x))\,dx>0,
\end{equation}
which contradicts \eqref{eqn: integrating eqn}. Therefore when $\theta\leq0$, $q(x)\geq \lambda_1$ for $x\in \mathbb{R}$. For the case where $\theta\geq0$, integrating \eqref{eqn: ODE for p and q<} with respect to $x$ from $z$ to $z_2$ yields
\begin{equation}\label{eqn: eqn by integrate from z to z2}
q'(z_2)-q'(z)+\theta\,(p(z_2)-p(z))+\int_{z}^{z_2}F(u(x),v(x))\,dx\leq0.
\end{equation}
In a similar manner, it can be shown that $q'(z_2)\ge 0$, $q'(z)=0$, $p(z_2)>\eta$, $p(z)<\eta$, and $\int_{z}^{z_2}F(u(x),v(x))\,dx>0$. These together contradict \eqref{eqn: eqn by integrate from z to z2}. Consequently, $(i)$ is proved for $d>1$.
For $d=1$, we have $q=p$ and \eqref{eqn: ODE for p and q<} becomes
\begin{equation}\label{eqn: ODE for p and q p=q}
p''(x)+\theta\,p'(x)+F(u(x),v(x))\leq0, \quad x\in \mathbb{R}.
\end{equation}
Moreover, when $d=1$ we take $\lambda_1=\lambda_2=\eta=\frac{\alpha}{a_2}$, i.e. the three lines $\alpha\,u+d\,\beta\,v=\lambda_1$, $\alpha\,u+d\,\beta\,v=\lambda_2$, and $\alpha\,u+\beta\,v=\eta$ coincide. Analogously to the case of $d>1$, we assume that there exists $\hat{z}\in\mathbb{R}$ such that $p(\hat{z})<\lambda_1$ and $\min_{x\in\mathbb{R}} p(x)=p(\hat{z})$. Due to $\min_{x\in\mathbb{R}} p(x)=p(\hat{z})$, we have $p'(\hat{z})=0$ and $p''(\hat{z})\geq0$. Since $(u(\hat z),v(\hat z))$ is in the interior of $\underaccent\bar{\mathcal{R}}$, which is contained in the interior of $A_+$, we have $F(u(\hat{z}),v(\hat{z}))>0$. These together give $p''(\hat{z})+\theta\,p'(\hat{z})+F(u(\hat{z}),v(\hat{z}))>0$, which contradicts \eqref{eqn: ODE for p and q p=q}. Thus, $p(x)\geq \lambda_1$ for all $x\in \mathbb{R}$ when $d=1$. As a result, the proof of $(i)$ is completed.


The proofs for cases $(ii)$, $(iii)$, and $(iv)$ are similar (see Figures~2.\ref{fig: d>1_beta_a_2d<alpha_a1}, 2.\ref{fig: d<1_beta_a_2d>alpha_a1}, and 2.\ref{fig: d<1_beta_a_2d<alpha_a1}). This completes the proof of Proposition~\ref{prop: lower bed}.
\end{proof}


\begin{figure}[ht!]
\centering
\mbox{
\subfigure[]{\includegraphics[width=0.52\textwidth]{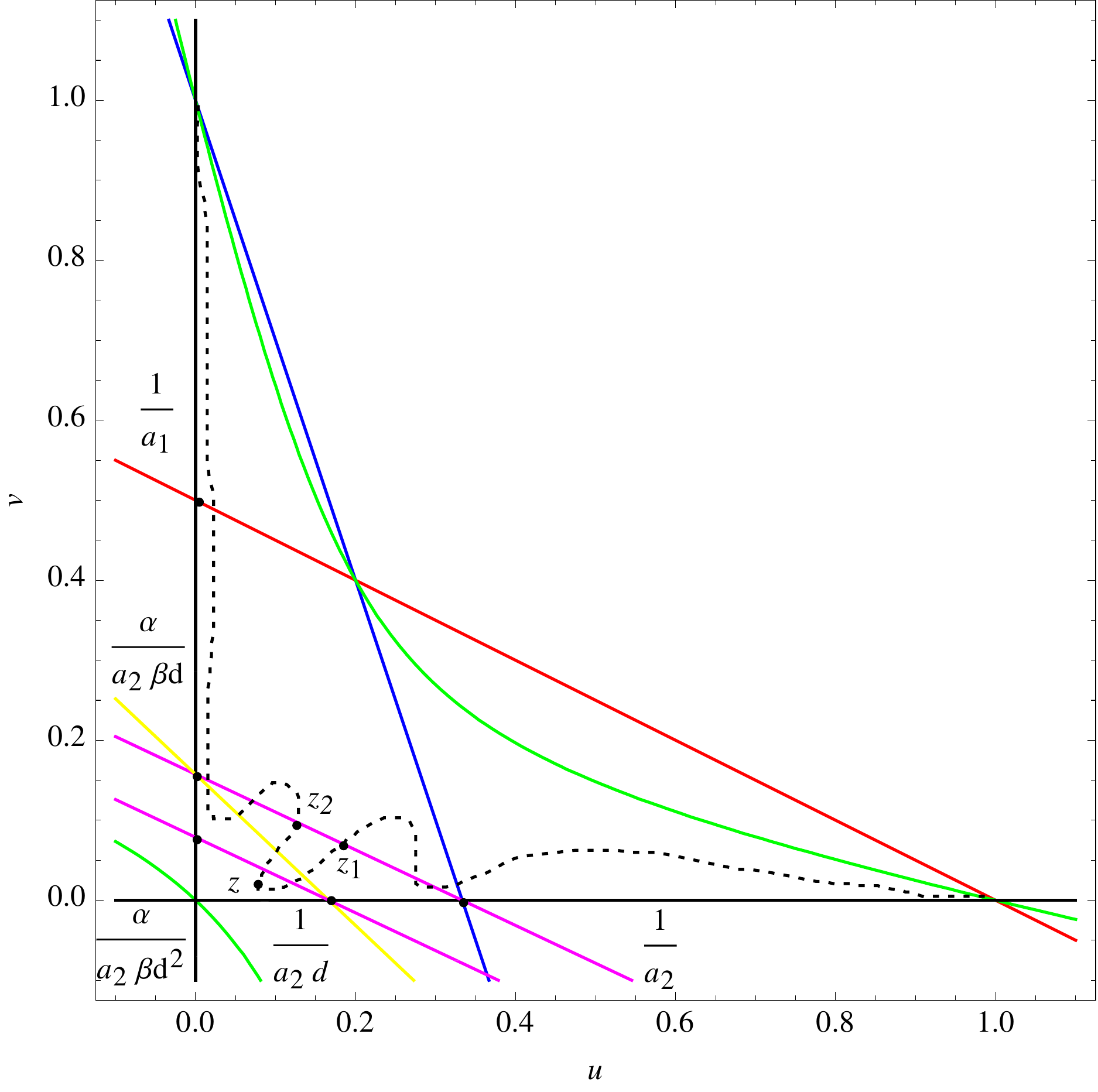}
             \label{fig: d>1_beta_a_2d>alpha_a1}    } \quad \hspace{0mm}
\subfigure[]{\includegraphics[width=0.52\textwidth]{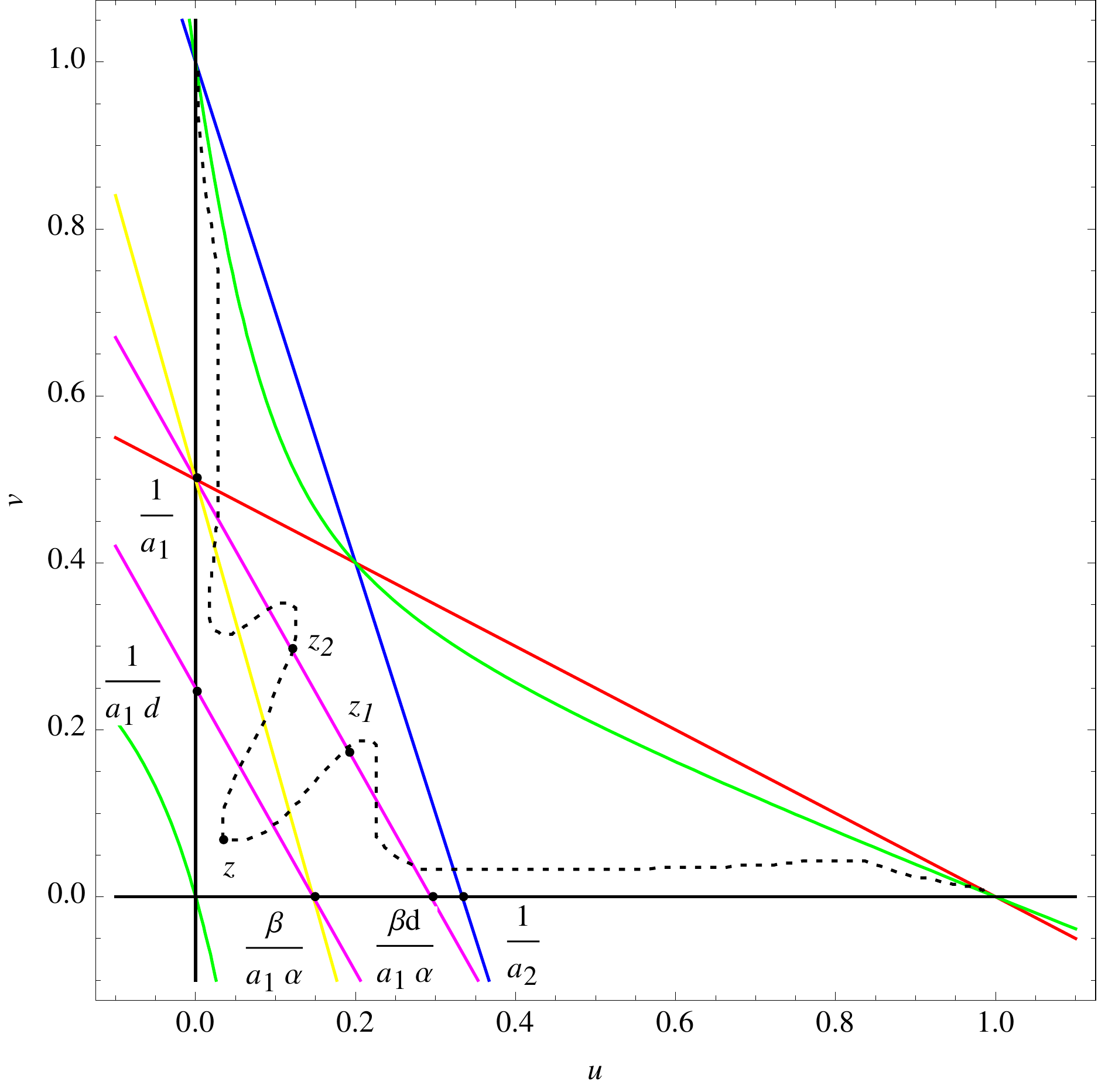}
             \label{fig: d>1_beta_a_2d<alpha_a1}    } \quad \hspace{0mm}
             }
\mbox{
\subfigure[]{\includegraphics[width=0.52\textwidth]{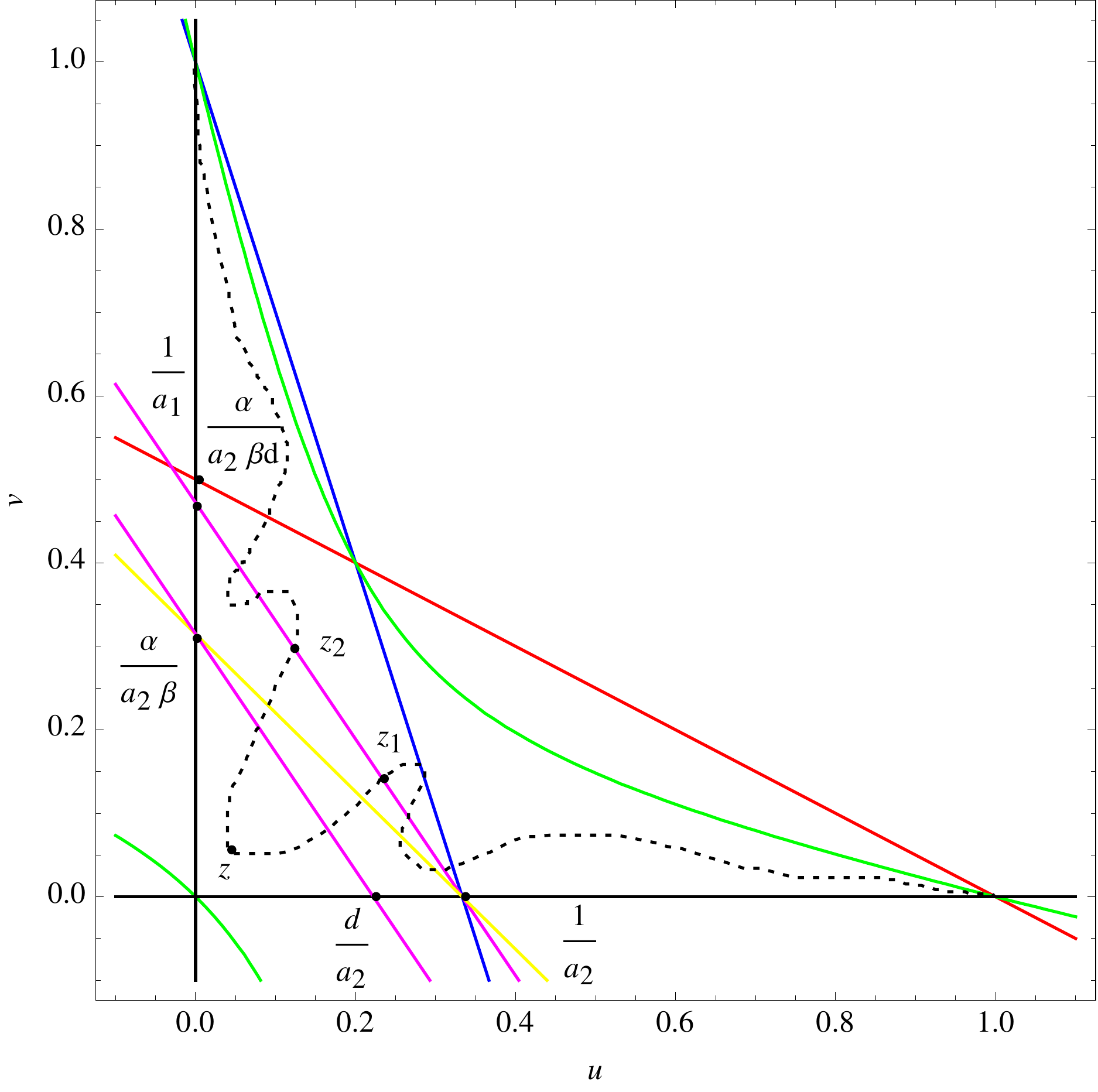}
             \label{fig: d<1_beta_a_2d>alpha_a1}    } \quad \hspace{0mm}
\subfigure[]{\includegraphics[width=0.52\textwidth]{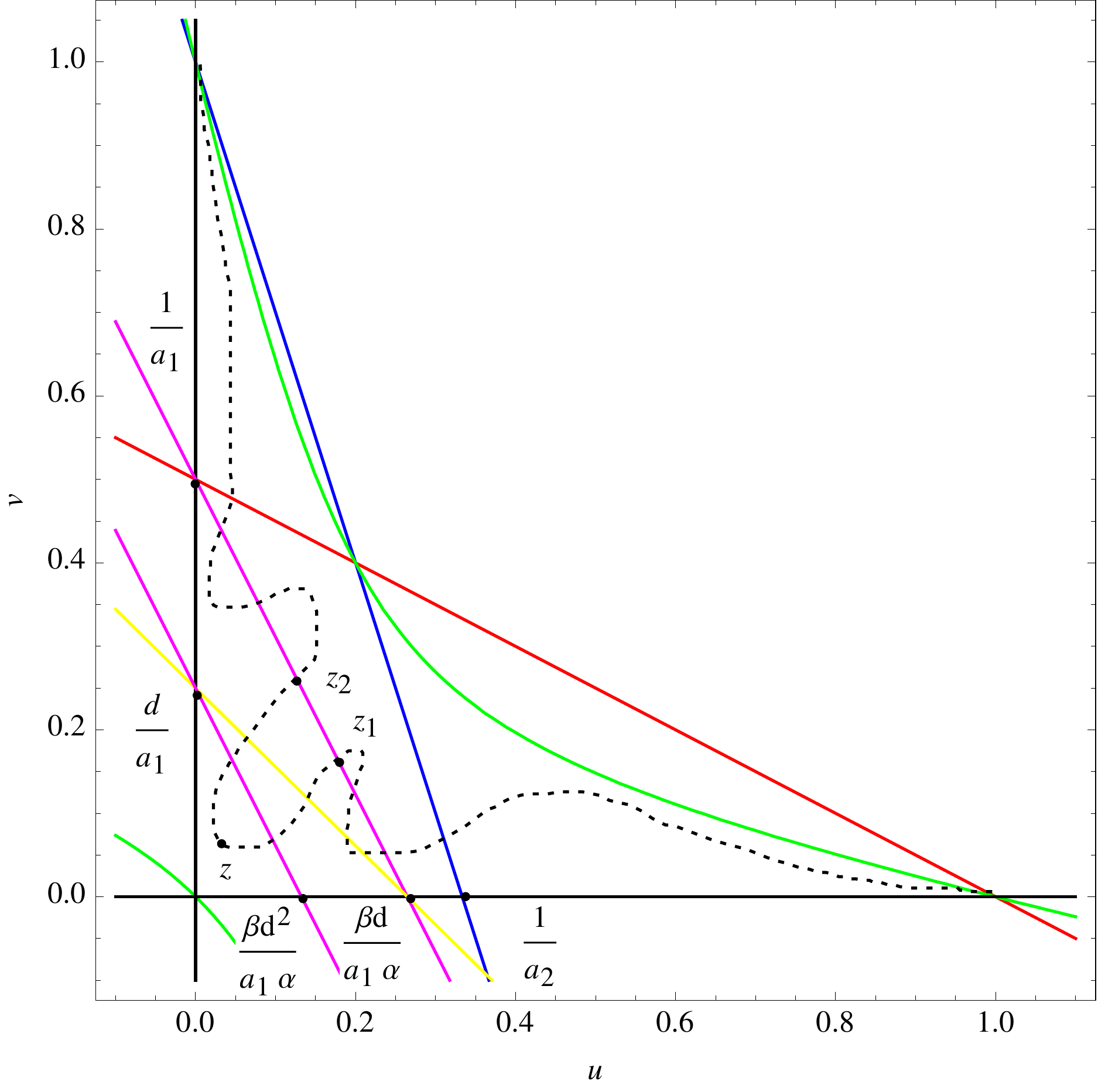}
             \label{fig: d<1_beta_a_2d<alpha_a1}    } \quad \hspace{0mm}
}
\caption{\small Red line: $1-u-a_1\,v=0$; blue line: $1-a_2\,u-v=0$; green curve: $\alpha\,u\,(1-u-a_1\,v)+\beta\,k\,v\,(1-a_2\,u-v)=0$; magenta line (above): $\alpha\,u+d\,\beta\,v=\lambda_2$; magenta line (below): $\alpha\,u+d\,\beta\,v=\lambda_1$; yellow line: $\alpha\,u+\beta\,v=\eta$; dashed curve: $(u(x),v(x))$.
\subref{fig: d>1_beta_a_2d>alpha_a1} $a_1=2$, $a_2=3$, $\alpha=17$, $\beta=18$, and $d=2$ give $\lambda_1=\frac{17}{6}$, $\lambda_2=\frac{17}{3}$, and $\eta=\frac{17}{6}$.
\subref{fig: d>1_beta_a_2d<alpha_a1} $a_1=2$, $a_2=3$, $\alpha=17$, $\beta=5$, and $d=2$ give $\lambda_1=\frac{5}{2}$, $\lambda_2=5$, and $\eta=\frac{5}{2}$.
\subref{fig: d<1_beta_a_2d>alpha_a1} $a_1=2$, $a_2=3$, $\alpha=17$, $\beta=18$, and $d=\frac{2}{3}$ give $\lambda_1=\frac{34}{9}$, $\lambda_2=\frac{17}{3}$, and $\eta=\frac{17}{3}$.
\subref{fig: d<1_beta_a_2d<alpha_a1} $a_1=2$, $a_2=3$, $\alpha=17$, $\beta=18$, and $d=\frac{1}{2}$ give $\lambda_1=\frac{9}{4}$, $\lambda_2=\frac{9}{2}$, and $\eta=\frac{9}{2}$.
\label{fig: 4 figures lower bound}}
\end{figure}

\begin{prop} [\textbf{Upper bound for $q=q(x)$}]\label{prop: upper bed}
Assume that $a_1>1$, $a_2>1$, and that $(u(x),v(x))$ is $C^2$, nonnegative, and satisfies the following differential inequalities:
\begin{equation}\label{eqn: L-V >0}
\begin{cases}
\vspace{3mm}
\hspace{2.0mm} u_{xx}+\theta\,u_{x}+\;\;u\,(1-u-a_1\,v)\geq0, \quad x\in\mathbb{R}, \\
\vspace{3mm}
d\,v_{xx}+\theta\,v_{x}+k\,v\,(1-a_2\,u-v)\geq0, \quad x\in\mathbb{R},\\
(u,v)(-\infty)=\text{\bf e}_2,\quad (u,v)(+\infty)= \text{\bf e}_3.
\end{cases}
\end{equation}
Then for $x\in\mathbb{R}$, we have


\begin{equation}
q(x)\leq \max\bigg[\frac{\alpha}{d},\beta\bigg]\,\max [1,d^2].
\end{equation}


\end{prop}

\begin{proof}
As in the proof of Proposition~\ref{prop: lower bed}, there are also four cases
and for each case, we can construct the N-barrier as shown
in Figures~2.\ref{fig: d>1_beta_d>alpha}, 2.\ref{fig: d>1_beta_d<alpha}, 2.\ref{fig: d<1_beta_d>alpha}, and 2.\ref{fig: d<1_beta_d<alpha}
and prove that $q(x)\leq \lambda_1$ for $x\in \mathbb{R}$:
\begin{itemize}
     \item If $d\geq1$,
\begin{itemize}
\item [$(i)$] when $\beta\,d\geq \alpha$, we take $(\lambda_1,\lambda_2,\eta):=(\beta\,d^2,\beta\,d,\beta\,d)$;

\item [$(ii)$] when $\beta\,d < \alpha$, $(\lambda_1,\lambda_2,\eta):=(\alpha\,d,\alpha,\alpha)$.
\end{itemize}
      \item If $d<1$,
\begin{itemize}
\item [$(iii)$] when $\beta\,d\geq \alpha$, $(\lambda_1,\lambda_2,\eta):=(\beta,\beta\,d,\beta)$;
\item [$(iv)$] when $\beta\,d < \alpha$, $(\lambda_1,\lambda_2,\eta):=(\dfrac{\alpha}{d},\alpha,\dfrac{\alpha}{d})$.
\end{itemize}


\end{itemize}
We note that cases $(i)$, $(ii)$, $(iii)$, and $(iv)$ corresponds to Figures~2.\ref{fig: d>1_beta_d>alpha}, 
2.\ref{fig: d>1_beta_d<alpha}, 2.\ref{fig: d<1_beta_d>alpha}, and 2.\ref{fig: d<1_beta_d<alpha}, respectively.
Combining the four cases above, it follows that
\begin{itemize}
  \item for $\beta\,d\geq \alpha$, $q(x)\leq \beta\max (1,d^2)$ for all $x\in \mathbb{R}$;
  \item for $\beta\,d<      \alpha$, $q(x)\leq \alpha\max (d,1/d)$ for all $x\in \mathbb{R}$,
\end{itemize}
which implies
$q(x)\leq \max[\frac{\alpha}{d},\beta]\,\max [1,d^2]$ for all $x\in \mathbb{R}$.
The rest part of the proof is similar to that of Proposition~\ref{prop: lower bed} 
and is hence omitted. 

\begin{figure}[ht!]
\centering
\mbox{
\subfigure[]{\includegraphics[width=0.53\textwidth]{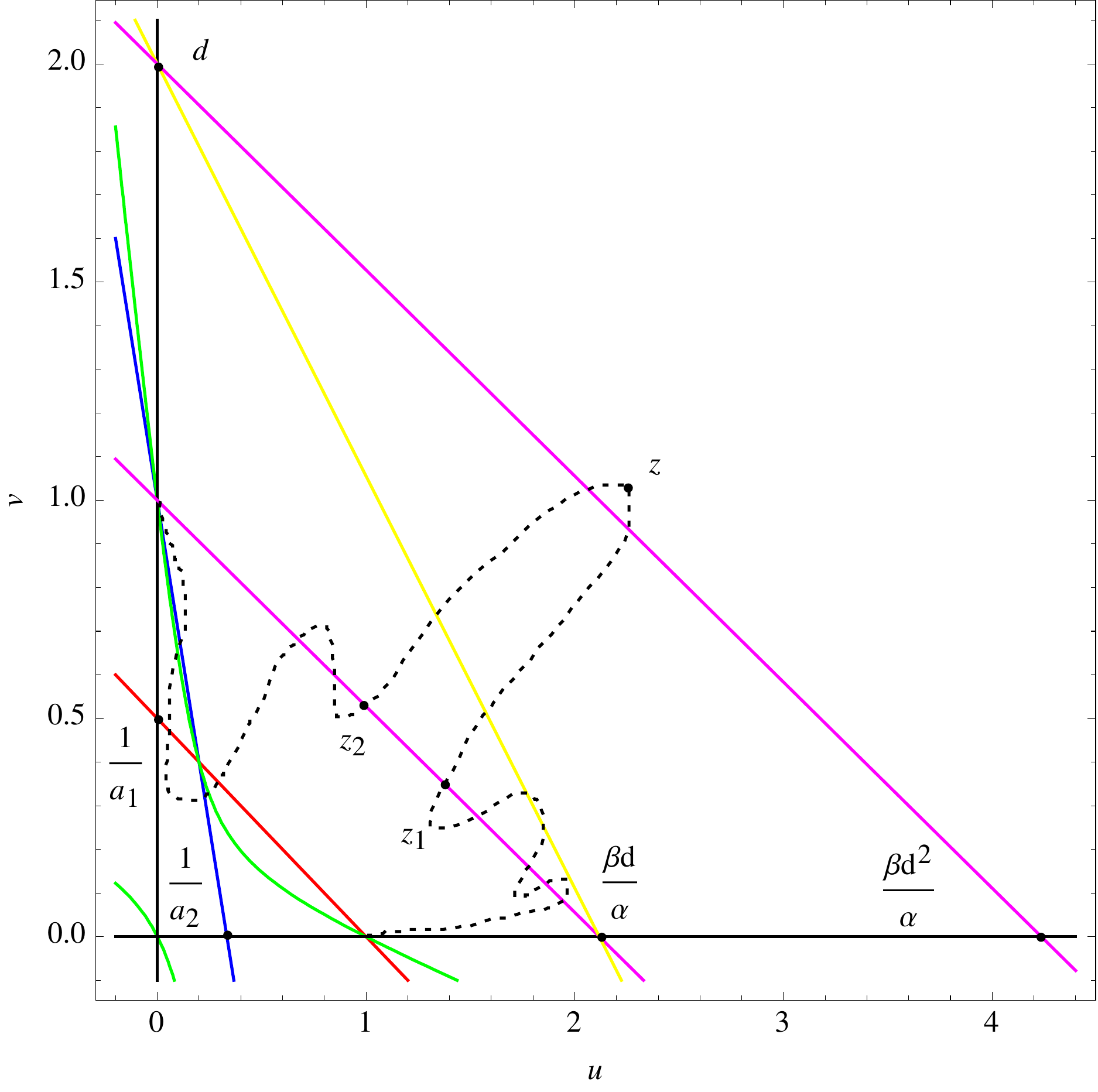}
             \label{fig: d>1_beta_d>alpha}    } \quad \hspace{0mm}
\subfigure[]{\includegraphics[width=0.53\textwidth]{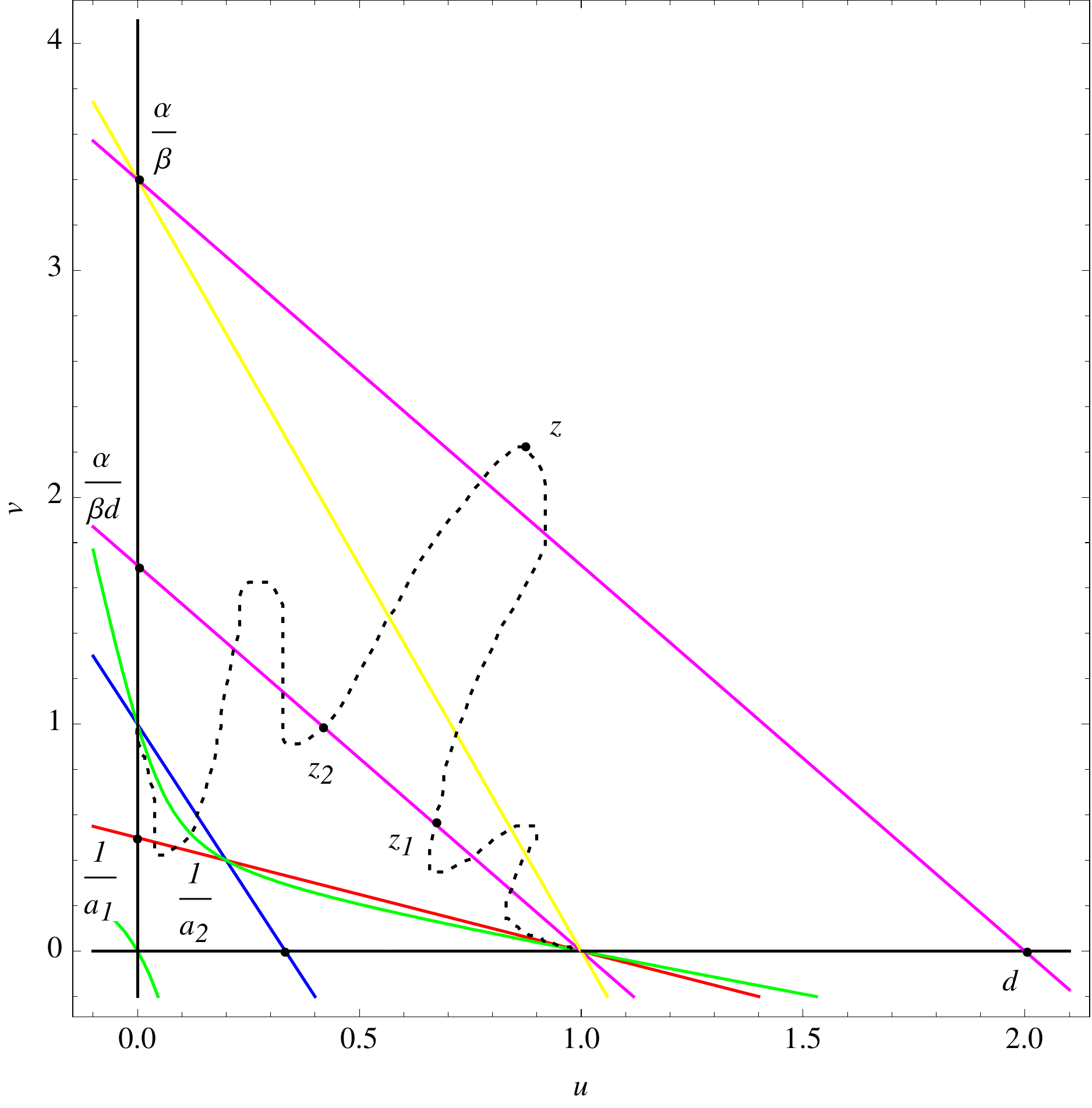}
             \label{fig: d>1_beta_d<alpha}    } \quad \hspace{0mm}
             }
\mbox{
\subfigure[]{\includegraphics[width=0.53\textwidth]{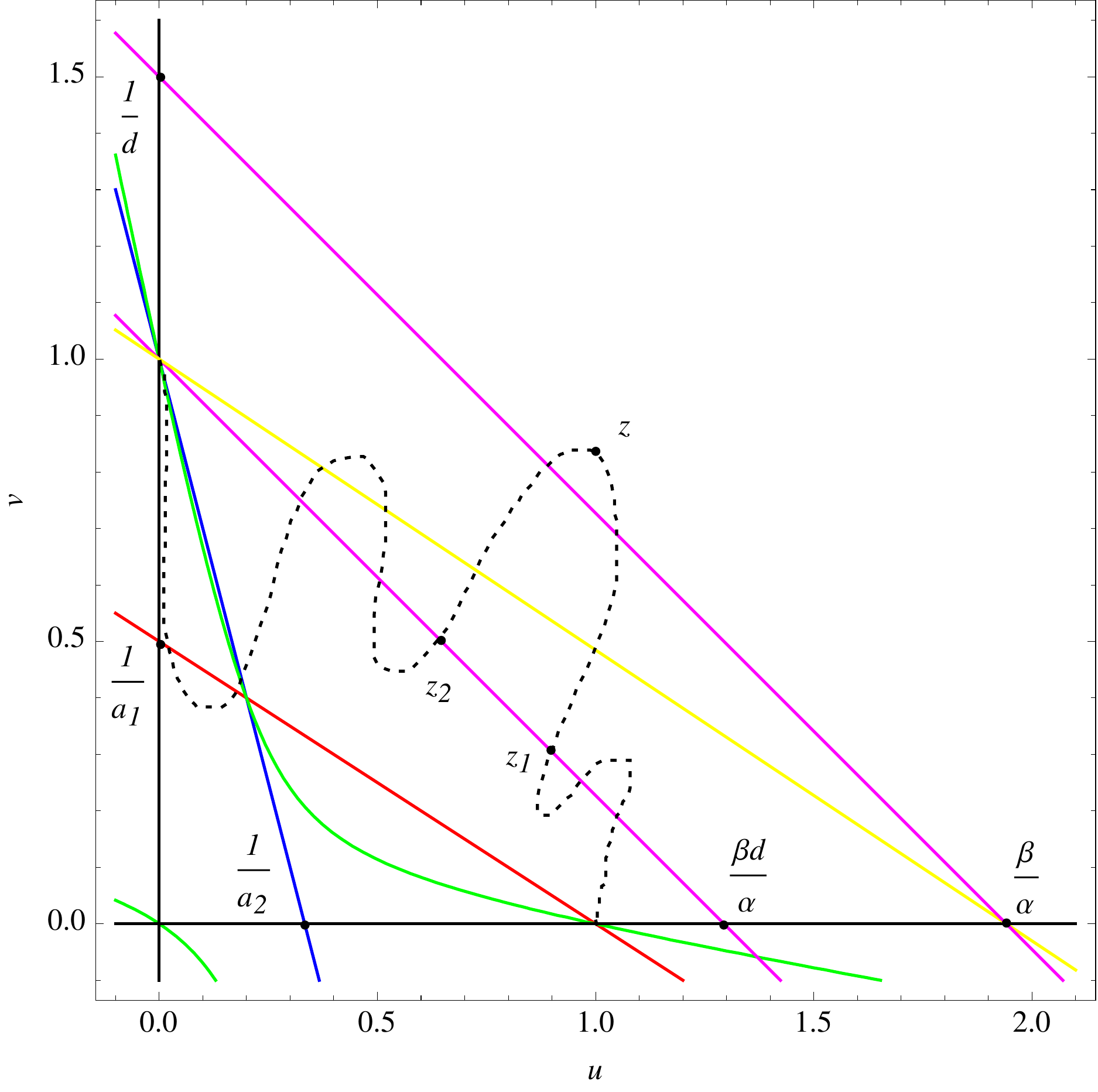}
             \label{fig: d<1_beta_d>alpha}    } \quad \hspace{0mm}
\subfigure[]{\includegraphics[width=0.53\textwidth]{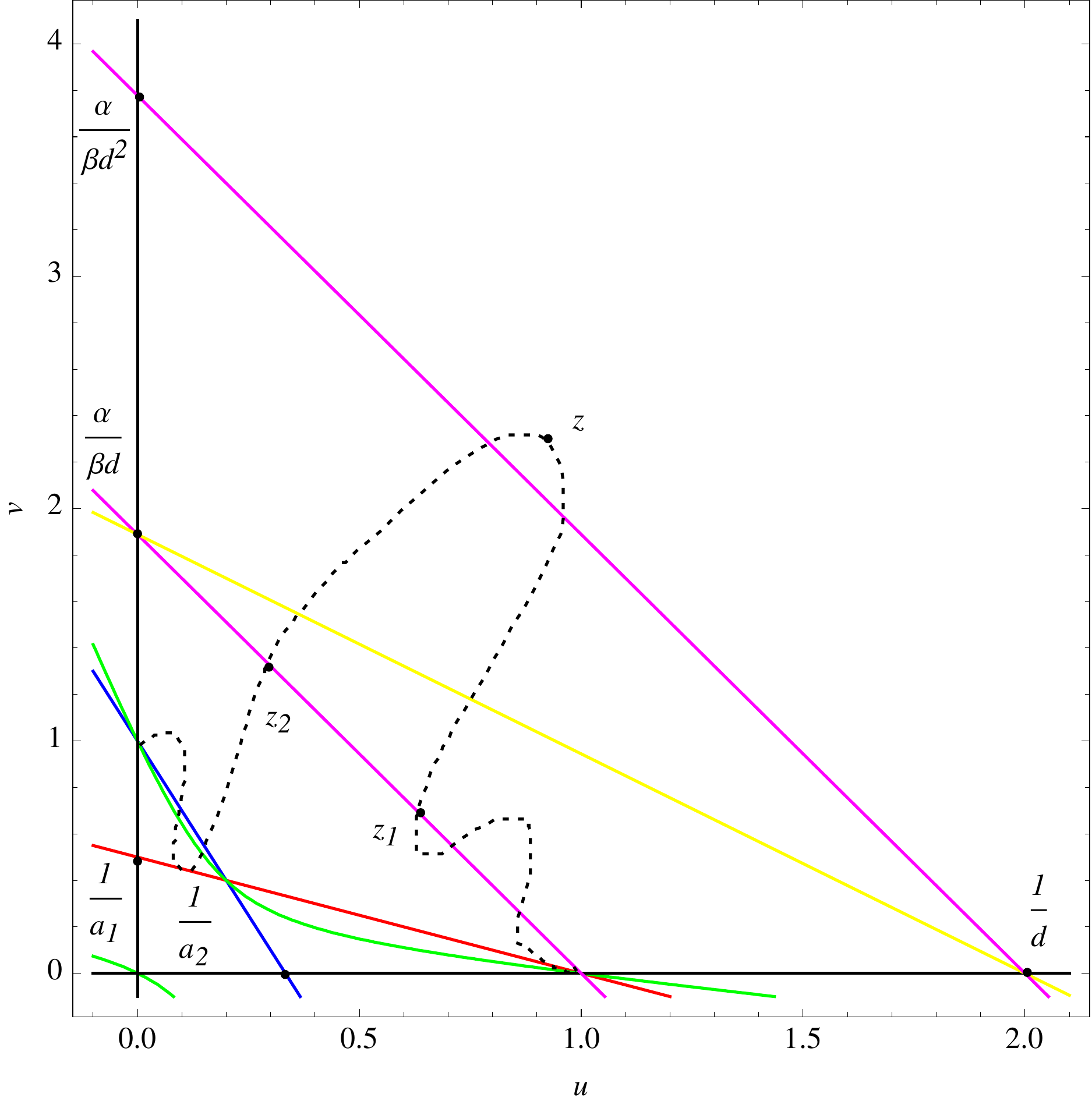}
             \label{fig: d<1_beta_d<alpha}    } \quad \hspace{0mm}
}
\caption{\small Red line: $1-u-a_1\,v=0$; blue line: $1-a_2\,u-v=0$; green curve: $\alpha\,u\,(1-u-a_1\,v)+\beta\,k\,v\,(1-a_2\,u-v)=0$; magenta line (below): $\alpha\,u+d\,\beta\,v=\lambda_2$; magenta line (above): $\alpha\,u+d\,\beta\,v=\lambda_1$; yellow line: $\alpha\,u+\beta\,v=\eta$; dashed curve: $(u(x),v(x))$.
\subref{fig: d>1_beta_d>alpha} $a_1=2$, $a_2=3$, $\alpha=17$, $\beta=18$, and $d=2$ give $\lambda_1=72$, $\lambda_2=36$, and $\eta=36$.
\subref{fig: d>1_beta_d<alpha} $a_1=2$, $a_2=3$, $\alpha=17$, $\beta=5$, and $d=2$ give $\lambda_1=34$, $\lambda_2=17$, and $\eta=17$.
\subref{fig: d<1_beta_d>alpha} $a_1=2$, $a_2=3$, $\alpha=17$, $\beta=33$, and $d=\frac{2}{3}$ give $\lambda_1=33$, $\lambda_2=22$, and $\eta=33$.
\subref{fig: d<1_beta_d<alpha} $a_1=2$, $a_2=3$, $\alpha=17$, $\beta=18$, and $d=\frac{1}{2}$ give $\lambda_1=34$, $\lambda_2=17$, and $\eta=34$.
\label{fig: 4 figures upper bound}}
\end{figure}

\end{proof}

\setcounter{equation}{0}
\setcounter{figure}{0}

\section{General maximum principle}\label{sec: General NBMP}
\vspace{2mm}

In this section, we prove Theorem \ref{thm: N-Barrier Maximum Principle for 2 Species}, which generalizes the maximum principle 
in Theorem~\ref{thm: lb<q(x)<ub} to elliptic systems with a wider class of nonlinear terms. 
Recall that $\bar{\mathcal{R}}=\big\{ (u,v)\;\big|\; \frac{u}{\bar{u}}+\frac{v}{\bar{v}}\ge 1,\; u,v\geq0 \big\}$ and
$\underaccent\bar{\mathcal{R}}=\big\{ (u,v)\;\big|\; \frac{u}{\underaccent\bar{u}}+\frac{v}{\underaccent\bar{v}}\le1,\;u,v\geq0\big\}$.

As in Section~\ref{sec: main result}, adding the two equations in \eqref{eqn: L-V BVP  before scaled-G} leads to an equation involving $p(x)=\alpha\,u(x)+\beta\,v(x)$ and $q(x)=d_1\,\alpha\,u(x)+d_2\,\beta\,v(x)$, i.e.
\begin{align}\label{eq: q''+p'+f+g=0 before scale}
  0&=\alpha\,\big(d_1\,u_{xx}+\theta\,u_{x}+u^m\,f(u,v)\big)+
      \beta\,\big(d_2\,v_{xx}+\theta\,v_{x}+v^n\,g(u,v)\big)\notag\\[3mm]
  &=q''(x)+\theta\,p'(x)+
  \alpha\,u^m\,f(u,v)+
    \beta\,v^n\,g(u,v)\notag\\[3mm]
  &=q''(x)+\theta\,p'(x)+F(u,v),
\end{align}
where $\alpha$, $\beta>0$ are \textit{arbitrary} constants and $F(u,v)=\alpha\,u^m\,f(u,v)+\beta\,v^n\,g(u,v)$.
By assumption $\mathbf{[A]}$, it readily follows that $F(u,v)\ge 0$ on $\underaccent\bar{\mathcal{R}}$ and $F(u,v)\le 0$ on $\bar{\mathcal{R}}$.
In Theorem \ref{thm: N-Barrier Maximum Principle for 2 Species},
when $\text{\bf e}_{-} = (0,0)$ or $\text{\bf e}_{+}= (0,0)$, the lower bound estimate
no longer holds but the upper bound estimate is still valid. In the following, we state a theorem, which is slightly more general
than Theorem \ref{thm: N-Barrier Maximum Principle for 2 Species}, to include
the upper bound estimate when $\text{\bf e}_{-} = (0,0)$ or $\text{\bf e}_{+}= (0,0)$.

\vspace{2mm}

\begin{thm}\label{thm: N-Barrier Maximum Principle for 2 Species_generalized}
Assume that $\mathbf{[A]}$ holds. If $a>0$, $b>0$, and $(u(x),v(x))$ is a nonnegative solution to \eqref{eqn: L-V BVP  before scaled-G}, then

\begin{equation}\label{eqn: upper and lower bounds of p_generalized}
\underaccent\bar{\lambda}
\leq a\, u(x)+b\, v(x) \leq
\bar{\lambda}, \quad x\in\mathbb{R},
\end{equation}
where
\begin{equation}
\bar{\lambda}=\max\big(a\,\bar{u},b\,\bar{v}\big)\,\frac{\max(d_1,d_2)}{\min(d_1,d_2)}
\end{equation}
and
\begin{equation}
\underaccent\bar{\lambda}=\min\big(a\,\underaccent\bar{u},b\,\underaccent\bar{v}\big)\,\frac{\displaystyle\min(d_1,d_2)}{\displaystyle\max(d_1,d_2)}\,\chi
.
\end{equation}
with $\chi$ defined by
\begin{equation}
\chi
=
\begin{cases}
\vspace{3mm}
0,
\quad \text{if} \quad  \text{\bf e}_{+}=(0,0) \quad \text{or} \quad \text{\bf e}_{-}=(0,0),\\
1,
\quad \text{otherwise}.
\end{cases}
\end{equation}

\end{thm}

\begin{proof}

Let $\alpha=a/d_1$ and $\beta=b/d_2$ in (\ref{eq: q''+p'+f+g=0 before scale}).
Then $q(x)=au(x)+bv(x)$ in (\ref{eq: q''+p'+f+g=0 before scale}). We employ \textit{the N-barrier method} developed in Section~\ref{sec: main result}
to show \eqref{eqn: upper and lower bounds of p_generalized}, which implies \eqref{eqn: upper and lower bounds of p}.

First, we assume $\text{\bf e}_{+}\neq (0,0)$ and $\text{\bf e}_{-}\neq (0,0)$. To construct an appropriate \textit{N-barrier} for the lower bound estimate, we consider $Q_{\lambda}=\{(u,v)\,|\,d_1\,\alpha\,u+d_2\,\beta\,v\le\lambda, u\ge 0, v\ge 0\}$ and
$P_{\eta}=\{(u,v)\,|\,\alpha\,u+\beta\,v\le\eta, u\ge 0, v\ge 0\}$, and chose $\lambda_1$, $\lambda_2$ and $\eta$ as large as possible such that
$Q_{\lambda_1}\subset P_{\eta}\subset Q_{\lambda_2}\subset \underaccent\bar{\mathcal{R}}$. By direct computation, $\lambda_1$, $\lambda_2$ and $\eta$
can be determined by
\begin{eqnarray}
\lambda_2 &=& \min\, (d_1\alpha\,\underaccent\bar{u}, d_2\alpha\,\underaccent\bar{v})\,=\,\min\big(a\underaccent\bar{u},b\underaccent\bar{v}\big), \\
\eta &=& \min\, (\frac1{d_1},\frac1{d_2})\,\lambda_2,\\
\lambda_1 &=& \min\, (d_1,d_2)\,\eta=\min\, (\frac{d_1}{d_2},\frac{d_2}{d_1})\,\lambda_2.
\end{eqnarray}
Since $F(u,v)\ge 0$ on $\underaccent\bar{\mathcal{R}}$, we can employ (\ref{eq: q''+p'+f+g=0 before scale}) and follow the arguments
in Section~\ref{sec: main result}
to obtain the lower bound estimate $q(x)\ge\lambda_1=\underaccent\bar{\lambda}$.

When the boundary conditions $\text{\bf e}_{+}=(0,0)$ or
$\text{\bf e}_{-}=(0,0)$, the argument in Section~\ref{sec: main result} can not be applied and only a trivial lower bound $a\,u(x)+b\,v(x)\geq0$ can be given.

The proof for the upper bound of $a\,u+b\,v$ is similar.
\end{proof}

\setcounter{equation}{0}
\setcounter{figure}{0}

\section{More delicate lower bound: tangent lines to quadratic curves}\label{sec: tangent lines}
\vspace{2mm}
In this section we provide an alternative approach to determine the line $\alpha\,u+d\,\beta\,v=\lambda_2$ in the proof of
Proposition~\ref{prop: lower bed} so that a bigger $\lambda_2$ can be chosen and a stronger lower bound for $q(x)$ can be given. To this end, we determine $\lambda_2$ by solving
\begin{subequations}\label{eqn: tangent line passing through (u,v,lambda2)}
\begin{equation}
\alpha\,u+d\,\beta\,v=\lambda_2,
\end{equation}
\begin{equation}
F(u,v)=0,
\end{equation}
\begin{equation}
\frac{F_u(u,v)}{\alpha}=\frac{F_v(u,v)}{d\,\beta},
\end{equation}
\end{subequations}
where $F_u(u,v)=\frac{\partial F}{\partial u}(u,v)$ and $F_v(u,v)=\frac{\partial F}{\partial v}(u,v)$.
In \eqref{eqn: tangent line passing through (u,v,lambda2)}, $\alpha\,u+d\,\beta\,v=\lambda_2$ is the tangent line to the quadratic curve
$F(u,v)=0$ and the line $\alpha\,u+d\,\beta\,v=\lambda_2$ is perpendicular to the vector $<F_u(u,v),F_v(u,v)>$. The solution $(u,v,\lambda_2)$ of \eqref{eqn: tangent line passing through (u,v,lambda2)} determines the point of tangency of the line $\alpha\,u+d\,\beta\,v=\lambda_2$ and the quadratic curve $F(u,v)=0$.

Equation \eqref{eqn: tangent line passing through (u,v,lambda2)} can be solved with the aid of Mathematica. It is easy to see that the first and third equations in \eqref{eqn: tangent line passing through (u,v,lambda2)} are linear, while the second one is quadratic. To solve \eqref{eqn: tangent line passing through (u,v,lambda2)}, we begin by solving the first and third equations to obtain $(u,v)=(u(\lambda_2),v(\lambda_2))$ as
\begin{subequations}\label{eqn: (u(lambda2),v(lambda2))}, namely
\begin{equation}
 u=\frac{-d \lambda_2 \left(\alpha a_1+a_2 \beta  k\right)+\alpha\beta  d (d-k)+2 \alpha  k \lambda_2}{2 \alpha  \left(-d \left(\alpha a_1+a_2 \beta  k\right)+\beta d^2+\alpha  k\right)},
\end{equation}
\begin{equation}
v=\frac{\alpha  a_1 \lambda_2 +\beta \left(a_2 k \lambda_2 +\alpha  d-2 d\lambda_2 -\alpha  k\right)}{2 \beta \left(\alpha  a_1 d+a_2 \beta  d k-\beta  d^2-\alpha  k\right)}.
\end{equation}
\end{subequations}
Substituting \eqref{eqn: (u(lambda2),v(lambda2))} into the second equation in \eqref{eqn: tangent line passing through (u,v,lambda2)} yields the following quadratic equation for $\lambda_2$:
\begin{equation}\label{eqn: quadratic equation for lambda2}
\mu_2\,\lambda_2^2+\mu_1\,\lambda_2+\mu_0=0,
\end{equation}
where
\begin{subequations}\label{eqn: 3 coeff in lambda2 eqn}
\begin{equation}
\mu_2=\frac{\left(\alpha  a_1+a_2 \beta k\right)^2-4 \alpha  \beta  k}{4 \alpha  \beta  \left(-d \left(\alpha a_1+a_2 \beta  k\right)+\beta d^2+\alpha  k\right)},
\end{equation}
\begin{equation}
\mu_1=\frac{-2 \alpha  \beta  \left((d+k)\left(\alpha  a_1+a_2 \beta k\right)-2 k (\alpha +\beta d)\right)}{4 \alpha  \beta  \left(-d \left(\alpha a_1+a_2 \beta  k\right)+\beta d^2+\alpha  k\right)},
\end{equation}
\begin{equation}
\mu_0=\frac{\alpha  \beta  (d-k)^2}{4\left(-d \left(\alpha a_1+a_2 \beta  k\right)+\beta d^2+\alpha  k\right)}.
\end{equation}
\end{subequations}
It follows from \eqref{eqn: quadratic equation for lambda2} that
\begin{equation}\label{eqn: lambda2 formula}
\lambda_2=\frac{-\mu _1\pm\sqrt{\mu _1^2-4 \mu _0\mu _2}}{2 \mu _2}.
\end{equation}
Using \eqref{eqn: 3 coeff in lambda2 eqn}, the discriminant $\mathcal{D}$ of \eqref{eqn: quadratic equation for lambda2} is given by
\begin{equation}
\mathcal{D}=\mu_1^2-4\mu_2\mu_0=\frac{k \left(-a_1 \alpha -a_2 \beta k+\alpha +\beta  k\right)}{-d\left(\alpha  a_1+a_2 \beta k\right)+\beta  d^2+\alpha  k}.
\end{equation}
To apply the approach proposed here, it is necessary that $\mathcal{D}\geq0$. In fact, $\mathcal{D}\neq0$ since $k \left(-a_1 \alpha -a_2 \beta k+\alpha +\beta  k\right)=k(\alpha(1-a_1)+k\beta(1-a_2))<0$. Moreover, it can be shown that $\mathcal{D}>0$ if and only if
\begin{equation}\label{eqn: condition for d}
\frac{\alpha  a_1+a_2 \beta  k-\sqrt{\left(\alpha a_1+a_2 \beta  k\right){}^2-4\alpha  \beta  k}}{2\beta }<d<\frac{\alpha  a_1+a_2 \beta  k+\sqrt{\left(\alpha a_1+a_2 \beta  k\right){}^2-4\alpha  \beta  k}}{2\beta }.
\end{equation}
Under the condition \eqref{eqn: condition for d}, $\mu_0<0$, $\mu_1>0$, and $\mu_2<0$ and hence the two roots $\lambda_2$ given by \eqref{eqn: lambda2 formula} are both positive. However, when
\begin{equation}
\lambda_2=\frac{-\mu _1+\sqrt{\mu _1^2-4 \mu _0\mu _2}}{2 \mu _2},
\end{equation}
it turns out that one of $u,v$ given by \eqref{eqn: (u(lambda2),v(lambda2))} is negative. We remark that this fact can also be easily seen from a property of the hyperbola $F(u,v)=0$. That is, for a given slope, there exist two tangent lines to the hyperbola $F(u,v)=0$: one has the intersection point in the first quadrant, while the other has the intersection point in the second or fourth quadrant. Therefore, we have
\begin{equation}\label{eqn: lambda2 Mathematica}
\lambda_2=\frac{-\mu _1-\sqrt{\mu _1^2-4 \mu _0\mu _2}}{2 \mu _2}
\end{equation}
and the intersection point $(u,v)$ can be expressed in terms of $\lambda_2$ using \eqref{eqn: (u(lambda2),v(lambda2))}. Now we are in a position to prove:

\vspace{2mm}

\begin{prop} [\textbf{Stronger lower bound for $q=q(x)$}]\label{prop: lower bed Mathematica}
Assume that $a_1>1$, $a_2>1$ and the condition \eqref{eqn: condition for d} holds. Let $\lambda_2$ be given by \eqref{eqn: lambda2 Mathematica}
and let $(u(x),v(x))$ be a pair of nonnegative $C^2$ functions satisfying the differential inequalities \eqref{eqn: L-V <0}. Then we have

\begin{itemize}
     \item when $d\geq1$,
     $q(x)\geq \,\dfrac{\lambda_2}{d}$ for all $x\in \mathbb{R}$;
     \item when $d<1 $,
     $q(x)\geq \,\lambda_2\,d$ for all $x\in \mathbb{R}$.
\end{itemize}
Consequently, for $d>0$ we have $q(x)\geq \lambda_2\,\min[d,1/d]$ for all $x\in \mathbb{R}$.
\end{prop}

\begin{proof}
The proof is similar to that of Proposition~\ref{prop: lower bed} (see Figure~\ref{fig: d>1_Mathematica} and Figure~\ref{fig: d<1_Mathematica}).
Here we only explain how to determine $(\lambda_1,\eta)$ (see Figure~\ref{fig: d<1_Mathematica}).
When $d<1$, we choose $(\lambda_1,\eta)=(\lambda_2\,d,\lambda_2)$ via the following procedure.
The line $\alpha\,u+d\,\beta\,v=\lambda_2$ intersects the $u$-axis and $v$-axis at $(\frac{\lambda_2}{\alpha},0)$ and $(0,\frac{\lambda_2}{\beta d})$,
respectively. Since $d<1$, the line $\alpha\,u+\beta\,v=\eta$ passes through the point $(\frac{\lambda_2}{\alpha},0)$, and $\eta$ is determined
by $\eta=\alpha\,u+\beta\,v\Big|_{(u,v)=\left(\frac{\lambda_2}{\alpha},0\right)}=\lambda_2$. The line $\alpha\,u+\beta\,v=\eta$ intersects the
$v$-axis at $(0,\frac{\lambda_2}{\beta})$, so $\lambda_1$ is given by
$\lambda_1=\alpha\,u+d\,\beta\,v\Big|_{(u,v)=\left(0,\frac{\lambda_2}{\beta}\right)}=\lambda_2\,d$.
For the case $d\geq1$, $\lambda_1$ and $\eta$ can be determined in a similar manner as $(\lambda_1,\eta)=(\frac{\lambda_2}{d},\frac{\lambda_2}{d})$.
(Refer to Figure~\ref{fig: d>1_Mathematica}.)

\end{proof}


Combining Propositions~\ref{prop: upper bed} and \ref{prop: lower bed Mathematica}, we obtain

\vspace{2mm}

\begin{cor}[\textbf{Maximum principle for $q(x)$}]\label{cor: lb<q(x)<ub Mathematica}
Assume that $a_1>1$, $a_2>1$ and the condition \eqref{eqn: condition for d} holds. Let $\lambda_2$ be given by \eqref{eqn: lambda2 Mathematica}
and let $(u(x),v(x))$ be a nonnegative solution to the following differential equations and asymptotic conditions
\begin{equation}\label{eqn: L-V >0 optimal lower bound}
\begin{cases}
\vspace{3mm}
u_{xx}+\theta\,u_{x}+u\,(1-u-a_1\,v)=0, \quad x\in\mathbb{R}, \\
\vspace{3mm}
v_{xx}+\theta\,v_{x}+v\,(1-a_2\,u-v)=0, \quad x\in\mathbb{R},\\
(u,v)(-\infty)=(1,0),\quad (u,v)(+\infty)= (0,1).
\end{cases}
\end{equation}
Then for $x\in\mathbb{R}$, we have
\begin{equation}
\frac{\lambda_2}{d}\,\min[1,d^2]\leq q(x)\leq \max\Big[\frac{\alpha}{d},\beta\Big]\max[1,d^2].
\end{equation}
\end{cor}

\vspace{2mm}

Compared with Theorem~\ref{thm: lb<q(x)<ub}, Corollary~\ref{cor: lb<q(x)<ub Mathematica} asserts that when \eqref{eqn: condition for d} holds, 
a stronger lower bound for $q(x)$ can be given in terms of $\lambda_2$, which is defined by \eqref{eqn: lambda2 Mathematica}. In particular, when $\alpha=\beta=d=k=1$, we obtain Theorem~\ref{thm: sharper lower bound for u+v}.


\vspace{5mm}
To illustrate Proposition~\ref{prop: lower bed Mathematica}, we give an example. When $a_1=2$, $a_2=3$, $\alpha=17$, $\beta=18$, $d=2$, and $k=1$, \eqref{eqn: tangent line passing through (u,v,lambda2)} can be solved to give
\begin{equation}
(u,v,\lambda_2)=\bigg(\frac{3\left(2349\pm71\sqrt{4611}\right)}{47270},\frac{17 \left(1131\pm4\sqrt{4611}\right)}{141810},\frac{153\left(79\pm\sqrt{4611}\right)}{1630}\bigg),
\end{equation}
which is approximately $(u,v,\lambda_2)=(-0.157,0.103,1.0415)$ or $(0.455,0.168,13.789)$. We choose $(u,v,\lambda_2)=\left(\frac{3\left(2349+71\sqrt{4611}\right)}{47270},\frac{17 \left(1131+4\sqrt{4611}\right)}{141810},\frac{153\left(79+\sqrt{4611}\right)}{1630}\right)\approx (0.455,0.168,13.789)$, and determine
\begin{equation}\label{eqn: lambda and eta}
\lambda_1=\eta=\frac{\lambda_2}{d}
\end{equation}
by employing Proposition~\ref{prop: lower bed Mathematica}. Then we are led to Figure~\ref{fig: d>1_Mathematica}.
Applying Proposition~\ref{prop: lower bed Mathematica} again, it follows that $q(x)\geq \lambda_1\approx6.895$ for all $x\in \mathbb{R}$
(see Figure~\ref{fig: d>1_Mathematica}). This lower bound is much bigger compared with the one given previously in Section~\ref{sec: main result}, where the lower bound for $q(x)$ is $\frac{17}{6}\approx2.833$ (see Figure~2.\ref{fig: d>1_beta_a_2d>alpha_a1}).



\begin{figure}[h]
\center \scalebox{0.6}{\includegraphics{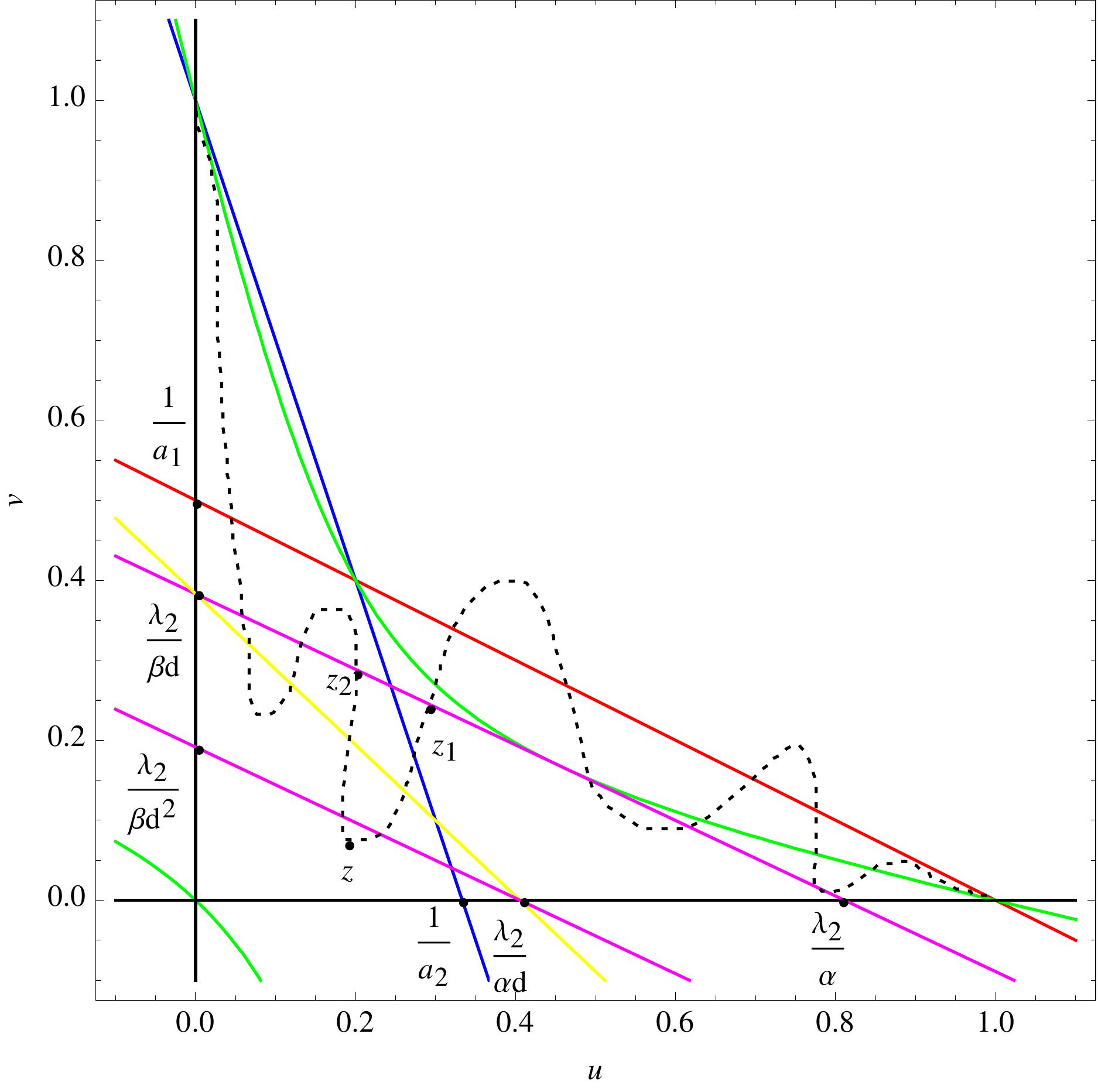}}
\caption{Red line: $1-u-a_1\,v=0$; blue line: $1-a_2\,u-v=0$; green curve: $\alpha\,u\,(1-u-a_1\,v)+\beta\,k\,v\,(1-a_2\,u-v)=0$; magenta line (above): $\alpha\,u+d\,\beta\,v=\lambda_2$; magenta line (below): $\alpha\,u+d\,\beta\,v=\lambda_1$; yellow line: $\alpha\,u+\beta\,v=\eta$; dashed curve: $(u(x),v(x))$. $a_1=2$, $a_2=3$, $\alpha=17$, $\beta=18$, $d=2$, $k=1$, and $\lambda_2=\frac{153\left(79+\sqrt{4611}\right)}{1630}\approx13.789$ give $\lambda_1=\eta=\frac{153\left(79+\sqrt{4611}\right)}{3260}\approx6.895$ according to \eqref{eqn: lambda and eta}
.}
\label{fig: d>1_Mathematica}
\end{figure}

\begin{figure}[h]
\center \scalebox{0.6}{\includegraphics{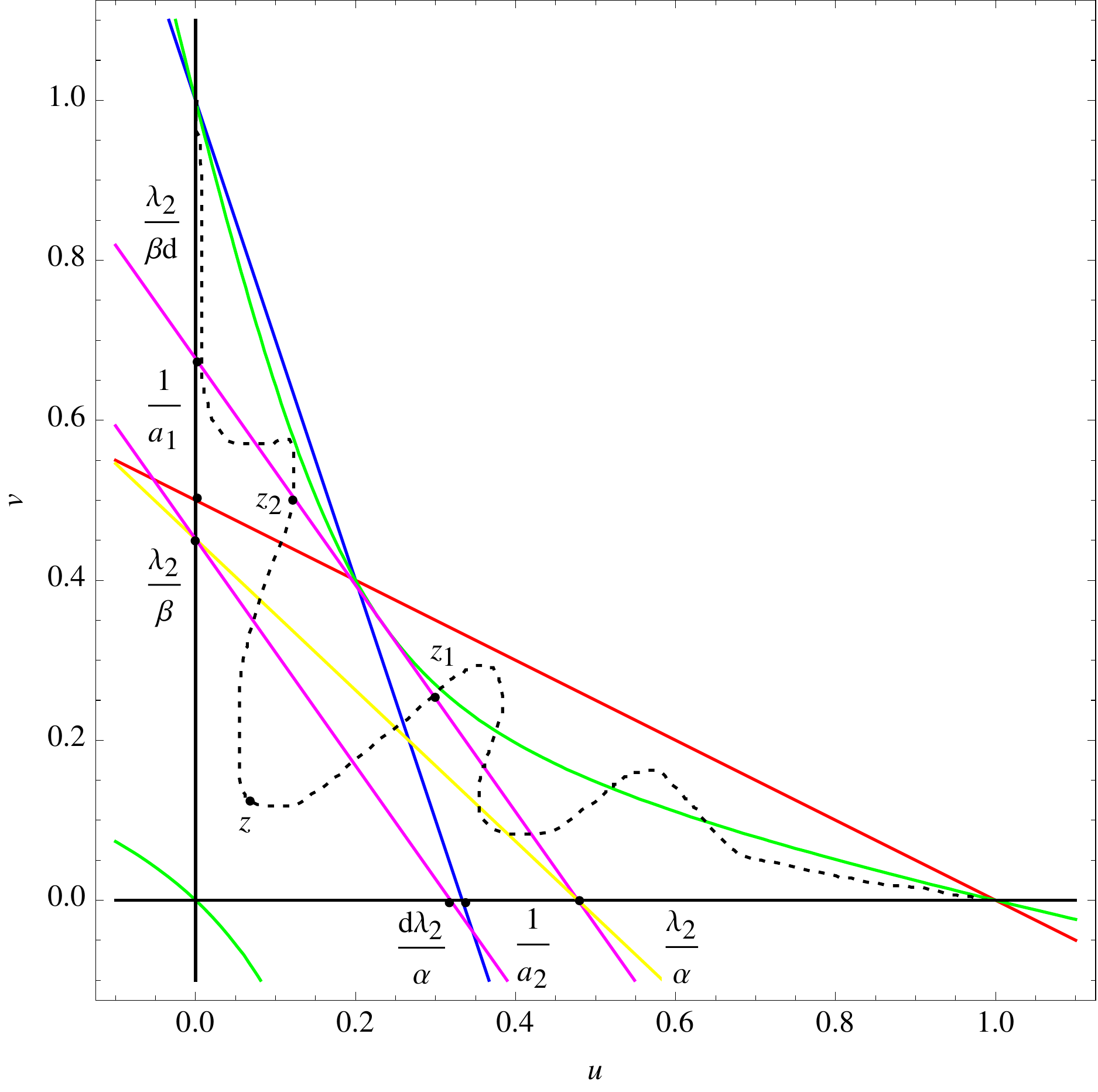}}
\caption{Red line: $1-u-a_1\,v=0$; blue line: $1-a_2\,u-v=0$; green curve: $\alpha\,u\,(1-u-a_1\,v)+\beta\,k\,v\,(1-a_2\,u-v)=0$; magenta line (upper): $\alpha\,u+d\,\beta\,v=\lambda_2$; magenta line (lower): $\alpha\,u+d\,\beta\,v=\lambda_1$; yellow line: $\alpha\,u+\beta\,v=\eta$; dashed curve: $(u(x),v(x))$. $a_1=2$, $a_2=3$, $\alpha=17$, $\beta=18$, $d=\frac{2}{3}$, $k=1$, and $\lambda_2=\frac{51\left(133+\sqrt{16059}\right)}{1630}\approx8.126$ give
$\lambda_1=\lambda_2\,d=\frac{17}{815}\left(133+\sqrt{16059}\right)\approx5.418$ and $\eta=\lambda_2=\frac{51\left(133+\sqrt{16059}\right)}{1630}\approx8.126$
.}
\label{fig: d<1_Mathematica}
\end{figure}




\vspace{2mm}
\setcounter{equation}{0}
\setcounter{figure}{0}
\section{Application to the nonexistence of three species travelling waves: proof of Theorem~\ref{thm: Nonexistence 3 species} }\label{sec: nonexistence}
\vspace{2mm}


In this section, we prove Theorem~\ref{thm: Nonexistence 3 species} by contradiction.
\begin{proof}[Proof of Theorem~\ref{thm: Nonexistence 3 species}]

Suppose to the contrary that there exists a solution $(u(x),v(x),w(x))$ to \eqref{eqn: L-V systems of three species (TWS)},\eqref{eqn: BC CHMU}.  
Due to the fact that $w(x)>0$ for $x\in\mathbb{R}$ and $w(\pm\infty)=0$, we can find $x_0\in\mathbb{R}$ 
such that $\max_{x\in\mathbb{R}} w(x)=w(x_0)>0$, $w''(x_0)\le 0$, and $w'(x_0)=0$. Since $w(x)$ satisfies $d_3\,w_{xx}+\theta \,w_x+w(\sigma_3-c_{31}\,u-c_{32}\,v-c_{33}\,w)=0$, we obtain
\begin{equation}\label{eqn: w(???) ???>0}
\sigma_3-c_{31}\,u(x_0)-c_{32}\,v(x_0)-c_{33}\,w(x_0)\ge 0,
\end{equation}
which gives
\begin{equation}
w(x)\leq w(x_0)\le \frac{1}{c_{33}}\big(\sigma_3-c_{31}\,u(x_0)-c_{32}\,v(x_0)\big)<\frac{\sigma_3}{c_{33}},\;x\in\mathbb{R}.
\end{equation}
As a consequence, we have
\begin{equation}\label{eqn: nonexistence diff ineq <0}
\begin{cases}
\vspace{3mm}
d_1\,u_{xx}+\theta \,u_x+u(\sigma_1-c_{13}\,\sigma_3\,c_{33}^{-1}-c_{11}\,u-c_{12}\,v)\leq0, \quad x\in\mathbb{R}, \\
d_2\,v_{xx}\hspace{0.8mm}+\theta \,v_x+v(\sigma_2-c_{23}\,\sigma_3\,c_{33}^{-1}-c_{21}\,u-c_{22}\,v)\leq0, \quad x\in\mathbb{R}.
\end{cases}
\end{equation}
Because of $\mathbf{[H1]}$ and $\mathbf{[H2]}$, we can apply Proposition~\ref{prop: lower bed, before scaling, appendix} from the Appendix to \eqref{eqn: nonexistence diff ineq <0}. Indeed, $\mathbf{[H1]}$ assures the positivity of $\sigma_1-c_{13}\,\sigma_3\,c_{33}^{-1}$ and $\sigma_2-c_{23}\,\sigma_3\,c_{33}^{-1}$, while the bistability condition $\mathbf{[BiS]}$ (see Appendix)  for the nonlinearity in \eqref{eqn: nonexistence diff ineq <0} follows from $\mathbf{[H2]}$. Consequently, we obtain a lower bound of $c_{31}\,u(x)+c_{32}\,v(x)$, i.e.
\begin{equation}
c_{31}\,u(x)+c_{32}\,v(x)\geq c_{33}^{-1}\,\min\bigg[\frac{\displaystyle c_{31}\,\phi_2}{\displaystyle c_{21}\,d_2},\frac{\displaystyle c_{32}\,\phi_1}{\displaystyle c_{12}\,d_1}\bigg]\,\min \big[d_1^2,d_2^2\big],\;x\in\mathbb{R}.
\end{equation}
The condition $\mathbf{[H3]}$ then yields
\begin{equation}
c_{31}\,u(x)+c_{32}\,v(x)\geq\sigma_3,\;x\in\mathbb{R},
\end{equation}
which contradicts \eqref{eqn: w(???) ???>0}. This completes the proof.

\end{proof}

\setcounter{equation}{0}
\setcounter{figure}{0}
\section{Appendix}\label{sec: appendix}
\vspace{2mm}

After suitable scaling, system \eqref{eqn: L-V BVP scaled} is equivalent to \eqref{eqn: L-V BVP before scaled}. Theorem \ref{thm: lb<q(x)<ub}
establishes lower-upper bound estimates for \eqref{eqn: L-V BVP scaled}. In this section, we state corresponding results
for \eqref{eqn: L-V BVP before scaled}. Throughout this section, we shall always assume the bistable condition:
\begin{itemize}
\item [$\mathbf{[BiS]}$] $\frac{\displaystyle\sigma_1}{\displaystyle c_{11}}>\frac{\displaystyle\sigma_2}{\displaystyle c_{21}}$, $\frac{\displaystyle\sigma_2}{\displaystyle c_{22}}>\frac{\displaystyle\sigma_1}{\displaystyle c_{12}}$,
\end{itemize}
which correponds to the condition $a_1>1$ and $a_2>1$ used in previous sections. Let $p(x)=\alpha\,u+\beta\,v$ and $q(x)=d_1\,\alpha\,u+d_2\,\beta\,v$. Then, from \eqref{eqn: L-V BVP  before scaled}, it follows that $p(x)$ and $q(x)$ satisfy
\begin{align}
  &\alpha\,\big(d_1\,u_{xx}+\theta\,u_{x}+u\,(\sigma_1-c_{11}\,u-c_{12}\,v)\big)+
      \beta\,\big(d_2\,v_{xx}+\theta\,v_{x}+v\,(\sigma_2-c_{21}\,u-c_{22}\,v)\big)\notag\\[3mm]
  &=q''(x)+\theta\,p'(x)+
  \alpha\,u\,(\sigma_1-c_{11}\,u-c_{12}\,v)+
    \beta\,v\,(\sigma_2-c_{21}\,u-c_{22}\,v).\notag\\
\end{align}
We can now apply the approach proposed in Section~\ref{sec: main result} to obtain lower and upper bounds for $q(x)$ in Proposition~\ref{prop: lower bed, before scaling, appendix} and Proposition~\ref{prop: upper bed, before scaling, appendix}, respectively.

\vspace{5mm}

\begin{prop} [\textbf{Lower bound for $q=q(x)$}]\label{prop: lower bed, before scaling, appendix}
Suppose that $(u(x),v(x))$ is $C^2$ and nonnegative, and satisfies the differential inequalities
\begin{equation}
\begin{cases}
\vspace{3mm}
d_1\,u_{xx}+\theta\,u_{x}+u\,(\sigma_1-c_{11}\,u-c_{12}\,v)\leq0, \quad x\in\mathbb{R}, \\
\vspace{3mm}
d_2\,v_{xx}\hspace{0.8mm}+\theta\,v_{x}+v\,(\sigma_2-c_{21}\,u-c_{22}\,v)\leq0,\quad x\in\mathbb{R}, \\
(u,v)(-\infty)=\big(\frac{\displaystyle\sigma_1}{\displaystyle c_{11}},0\big),\quad
(u,v)(+\infty)=\big(0,\frac{\displaystyle\sigma_2}{\displaystyle c_{22}}\big).
\end{cases}
\end{equation}
Then for $x\in\mathbb{R}$, we have
\begin{equation}
q(x)\geq \min\bigg[\frac{\alpha\,\sigma_2}{c_{21}\,d_2},\frac{\beta\,\sigma_1}{c_{12}\,d_1}\bigg]\,\min \big[d_1^2,d_2^2\big].
\end{equation}
\end{prop}


\vspace{5mm}

\begin{prop} [\textbf{Upper bound for $q=q(x)$}]\label{prop: upper bed, before scaling, appendix}
Suppose that $(u(x),v(x))$ is $C^2$ and nonnegative, and satisfies the following differential inequalities and asymptotic conditions
\begin{equation}
\begin{cases}
\vspace{3mm}
d_1\,u_{xx}+\theta\,u_{x}+u\,(\sigma_1-c_{11}\,u-c_{12}\,v)\geq0, \quad x\in\mathbb{R}, \\
\vspace{3mm}
d_2\,v_{xx}\hspace{0.8mm}+\theta\,v_{x}+v\,(\sigma_2-c_{21}\,u-c_{22}\,v)\geq0,\quad x\in\mathbb{R}, \\
(u,v)(-\infty)=\big(\frac{\displaystyle\sigma_1}{\displaystyle c_{11}},0\big),\quad
(u,v)(+\infty)=\big(0,\frac{\displaystyle\sigma_2}{\displaystyle c_{22}}\big).
\end{cases}
\end{equation}
Then for $x\in\mathbb{R}$, we have
\begin{equation}
q(x)\leq \max\bigg[\frac{\alpha\,\sigma_1}{c_{11}\,d_2},\frac{\beta\,\sigma_2}{c_{22}\,d_1}\bigg]\,\max \big[d_1^2,d_2^2\big].
\end{equation}

\end{prop}

\vspace{2mm}

\begin{thm}[\textbf{Maximum principle for $q(x)$}]\label{thm: lb<q(x)<ub, before scaling, appendix}
Suppose that $(u(x),v(x))$ is a nonnegative solution to the differential equations
\begin{equation}
\begin{cases}
\vspace{3mm}
d_1\,u_{xx}+\theta\,u_{x}+u\,(\sigma_1-c_{11}\,u-c_{12}\,v)=0, \quad x\in\mathbb{R}, \\
\vspace{3mm}
d_2\,v_{xx}\hspace{0.8mm}+\theta\,v_{x}+v\,(\sigma_2-c_{21}\,u-c_{22}\,v)=0,\quad x\in\mathbb{R}, \\
(u,v)(-\infty)=\big(\frac{\displaystyle\sigma_1}{\displaystyle c_{11}},0\big),\quad
(u,v)(+\infty)=\big(0,\frac{\displaystyle\sigma_2}{\displaystyle c_{22}}\big).
\end{cases}
\end{equation}
Then for $x\in\mathbb{R}$, we have
\begin{equation}\label{eqn: lower and upper bound in Appendix}
\min\bigg[\frac{\alpha\,\sigma_2}{c_{21}\,d_2},\frac{\beta\,\sigma_1}{c_{12}\,d_1}\bigg]\,\min \big[d_1^2,d_2^2\big]\leq q(x)\leq \max\bigg[\frac{\alpha\,\sigma_1}{c_{11}\,d_2},\frac{\beta\,\sigma_2}{c_{22}\,d_1}\bigg]\,\max \big[d_1^2,d_2^2\big].
\end{equation}
\end{thm}

\begin{proof}
Combining Proposition~\ref{prop: lower bed, before scaling, appendix} and Proposition~\ref{prop: upper bed, before scaling, appendix}, Theorem~\ref{thm: lb<q(x)<ub, before scaling, appendix} is established.
\end{proof}

We note in particular that, when $\alpha=d_1^{-1}$ and $\beta=d_2^{-1}$, \eqref{eqn: lower and upper bound in Appendix} becomes
\begin{equation}
\min\bigg[\frac{\sigma_2}{c_{21}},\frac{\sigma_1}{c_{12}}\bigg]\,\min \bigg[\frac{d_1}{d_2},\frac{d_2}{d_1}\bigg]\leq u(x)+v(x)\leq \max\bigg[\frac{\sigma_1}{c_{11}},\frac{\sigma_2}{c_{22}}\bigg]\,\max \bigg[\frac{d_1}{d_2},\frac{d_2}{d_1}\bigg].
\end{equation}
The above result can be generalised by letting $\alpha=d_1^{-1}\,r_1$ and $\beta=d_2^{-1}\,r_2$. Equation \eqref{eqn: lower and upper bound in Appendix} then leads to
\begin{eqnarray}
\min\bigg[\dfrac{r_1\sigma_2}{c_{21}},\dfrac{r_2\sigma_1}{c_{12}}\bigg]\min \bigg[\dfrac{d_1}{d_2},\dfrac{d_2}{d_1}\bigg]&\leq& r_1u(x)+r_2v(x)\nonumber\\
&\leq&
\max\bigg[\dfrac{r_1\sigma_1}{c_{11}},\dfrac{r_2\sigma_2}{c_{22}}\bigg]\max \bigg[\dfrac{d_1}{d_2},\dfrac{d_2}{d_1}\bigg],
\end{eqnarray}
where $r_1,r_2>0$ are arbitrary constants.


\vspace{2mm}

\textbf{Acknowledgments.} The authors wish to express sincere gratitude to
 Dr. Tom Mollee for his careful reading of the manuscript and valuable suggestions and comments
to improve the readability and accuracy of the paper. CC Chen are grateful for the
support by the grant 102-2115-M-002-011-MY3 of Ministry of Science and Technology, Taiwan. The research of L.-C. Hung is partly supported by the grant 104EFA0101550 of Ministry of Science and Technology, Taiwan.




\vspace{2mm}











\end{document}